\documentclass[11pt,reqno]{amsart}

\usepackage[american]{babel}
\usepackage{todonotes}
\usepackage{comment}
\usepackage{amsmath}
\usepackage{dsfont,mathtools,amssymb}
\usepackage{enumerate}
\usepackage{color}
\usepackage[hidelinks]{hyperref}
\mathtoolsset{showonlyrefs,showmanualtags}
\usepackage{xfrac}

\newtheorem{theorem}{Theorem}[section]
\newtheorem{proposition}[theorem]{Proposition}
\newtheorem{lemma}[theorem]{Lemma}
\newtheorem{corollary}[theorem]{Corollary}

\numberwithin{equation}{section}
\numberwithin{figure}{section}

\newcommand{\E}{\mathds{E}}
\renewcommand{\P}{\mathds{P}}

\renewcommand{\Pr}{\mathbf{P}}
\newcommand{\R}{\mathbb{R}}
\newcommand{\T}{\mathbb{T}}

\newcommand{\Z}{\mathbb{Z}}
\newcommand{\N}{\mathbb{N}}
\newcommand{\dd}{{\rm d}}
\newcommand{\fstop}{\; \text{.}}
\newcommand{\comma}{\; \text{,}\;\;}

\newcommand{\seq}[1]{\left(#1\right)}

\newcommand{\tonde}[1]{\left(#1\right)}
\newcommand{\ttonde}[1]{\big(#1\big)}
\newcommand{\tttonde}[1]{(#1)}
\newcommand{\quadre}[1]{\left[#1\right]}
\newcommand{\abs}[1]{\left\lvert#1\right\rvert}

\newcommand{\emparg}{\,\cdot\,}

\newcommand{\eqdef}{\coloneqq}
\newcommand{\defeq}{\eqqcolon}
\newcommand{\car}{\mathds{1}}

\newcommand{\norm}[1]{\left\lVert#1\right\rVert}

\newcommand{\eps}{\varepsilon}

\newcommand{\set}[1]{\left\{#1\right\}}							

\newcommand{\cC}{\ensuremath{\mathcal C}} 
 
\newcommand{\cE}{\ensuremath{\mathcal E}}

\newcommand{\cL}{\ensuremath{\mathcal L}} 
 
\newcommand{\cN}{\ensuremath{\mathcal N}} 
 
\newcommand{\cP}{\ensuremath{\mathcal P}}

\allowdisplaybreaks
\title[Concentration and local smoothness]{Concentration and local smoothness\\ of the averaging process}

\subjclass[2020]{Primary 60K35; secondary 60J27, 82B20, 82C26, 91D30.}
\author{Federico Sau}
\address{Dipartimento di Matematica, Informatica e Geoscienze,
	Università degli Studi di Trieste, Trieste, Italy}
\email{federico.sau@units.it}

\keywords{Interacting particle systems; mixing of Markov chains; averaging process; hydrodynamic limits}
\begin{document} 
	\maketitle	
\begin{abstract}
We consider the averaging process
on the discrete $d$-dimensional torus. On this graph, the process is known to converge to equilibrium on diffusive timescales, not exhibiting cutoff. In this work, we refine this picture in two ways. Firstly, we prove a concentration phenomenon of the averaging process around its mean, occurring on a shorter timescale than the one of its relaxation to equilibrium. Secondly, we establish sharp gradient estimates, which capture its fast local smoothness property. This is the first setting in which these two features of the averaging process --- concentration and local smoothness --- can be quantified. These features carry useful information on a number of large scale properties of the averaging process. As an illustration of this fact, we determine the limit profile of its distance to equilibrium and derive a quantitative hydrodynamic limit for it. Finally, we discuss their implications on cutoff for the binomial splitting process, the particle analogue of the averaging process.
\end{abstract}

\section{Introduction}
The \emph{averaging process} on a graph is a random evolution of  a probability mass function over its vertices. Its dynamics consists in repeatedly iterating the following two operations:
\begin{enumerate}[(i)]
	\item \label{it:intro-1} first, select an unordered pair of nearest neighbor vertices at the arrival times of i.i.d.\ Poisson clocks;
	\item \label{it:intro-2} then, level out the masses associated to them.
\end{enumerate}
Hence, step \eqref{it:intro-1} is random, while step \eqref{it:intro-2} is deterministic. In particular, each update is performed at some i.i.d.\  space-time locations, and conserves the total mass. Therefore,  as time runs, for any mass initialization and provided the underlying graph is connected, the averaging dynamics will converge to a deterministic flat distribution.   
In view of the degeneracy of the equilibrium and of the intrinsic irreversibility of the averaging dynamics, a comprehensive quantitative picture  of such a convergence in relation to the underlying graph structure is, despite the recent advances, still a largely open problem of both applied and theoretical interest.	

The averaging process  appears in the computer science literature as a basic model of opinion dynamics (or randomized distributed/gossip algorithms, see, e.g., \cite{boyd_et_al_randomized_2006,shah2009gossip,movassagh_repeated2022} and references therein), together with a long list of variants, e.g.,  with updates involving only one vertex at a time (see, e.g., \cite{becchetti_consensus_2020} and references therein). In this realm, vertices are typically interpreted as agents, their masses as corresponding to their opinions, and the long-run limiting distribution arising from the local random interactions is then referred to as the consensus. 

The  mathematical interest on the averaging process as an interacting particle system  was revived about a decade ago by a series of lectures and expository articles by Aldous \cite{aldous2011finite} and Aldous and Lanoue \cite{aldous_lecture_2012}. Since then, the list of works on the subject rapidly grew (see, e.g., \cite{chatterjee2020phase,spiro_averaging_2022,quattropani2021mixing,movassagh_repeated2022,cao_explicit_2023,caputo_quattropani_sau_cutoff_2023}). In particular,  the seminal paper \cite{aldous_lecture_2012} was the first one to link this model to the theory of Markov chains mixing times \cite{aldous-fill-2014,levin2017markov,montenegro_mathematical_2005}, and  established the first general bounds on the mixing of the averaging process. Only recently,   \cite{chatterjee2020phase,quattropani2021mixing,caputo_quattropani_sau_cutoff_2023} provided the first  sharp asymptotic mixing results  on specific geometries, either proving or excluding the occurrence of the so-called cutoff phenomenon \cite[Chapter 18]{levin2017markov}. 

In this article, we consider the averaging process on large discrete $d$-dimensional tori. Mixing in this setting is covered by the analysis carried out in \cite{quattropani2021mixing}, which shows that the averaging process mixes gradually (i.e., without exhibiting an abrupt cutoff	 phenomenon) on  diffusive timescales. Our aim is to refine this picture, by analyzing two new features of the averaging process --- an early concentration phenomenon, and a fast local smoothness property --- and extract from them quantitative information on its mixing and scaling behaviors. After a quick recap on the averaging process and its main properties, we  present our findings first informally in Section \ref{sec:summary-results}, and then in detail in Section \ref{sec:results}.

\subsection{Setting and model definition}
For every integer $d\ge 1$, let $\T^d_N\eqdef \seq{\Z/N\Z}^d$ be the discrete $d$-dimensional torus of size $N\in \N$. Moreover, let $\cP(\T^d_N)$ denote the space of probability mass functions on $\T^d_N$, namely,
\begin{equation}
	\cP(\T^d_N)\eqdef \bigg\{\eta\in [0,1]^{\T^d_N} \,\bigg|\, \sum_{x\in \T^d_N}\eta(x)=1\bigg\}\fstop
\end{equation} 
For such a compact subset $\cP(\T^d_N)$  of $\R^{\T^d_N}$, we further let 
$\cC(\cP(\T^d_N))$ indicate the Banach space of continuous functions on $\cP(\T^d_N)$ endowed with the uniform norm.

The \textit{averaging process on $\T^d_N$} (shortly, ${\rm Avg}(\T^d_N)$) is the Markov jump process evolving on  $\cP(\T^d_N)$ and with infinitesimal generator $\cL=\cL_N$ given, for all $f\in \cC(\cP(\T^d_N))$, as 
\begin{equation}\label{eq:gen-avg}
	\cL f(\eta)\eqdef \frac12 \sum_{x\in \T^d_N}\sum_{\substack{y\in \T^d_N\\
			\abs{x-y}=1}}\big(f(\eta^{xy})-f(\eta)\big)\comma\qquad \eta \in \cP(\T^d_N)\fstop
\end{equation}
Here, $\abs{x-y}$ denotes the usual graph distance between vertices $x, y\in  \T^d_N$, 	while $\eta^{xy}\in \cP(\T^d_N)$ is obtained from $\eta\in \cP(\T^d_N)$ by updating the masses $\eta(x)$ and $\eta(y)$ with their average value: for all $x, y$ and $z \in \T^d_N$, 
\begin{align}\label{eq:eta-xy}
	\eta^{xy}(z)\eqdef \begin{dcases}
		\tfrac12 \tonde{\eta(x)+\eta(y)} &\text{if}\ z=x\ \text{or}\ z=y\\
		\eta(z) &\text{else}\fstop
	\end{dcases}
\end{align}

Henceforth,  although the updates occur at both random times and locations, as time runs,  the deterministic averages and the conservation of mass lead the (random) distribution  to become flat, for any mass initialization. In formulas,  this reads as
\begin{align}\label{eq:conv-qualitative}
	\eta^\xi_t\xrightarrow{t\to \infty}\pi\equiv 1/N^d\comma\qquad \xi \in \cP(\T^d_N)\comma
\end{align}
where $(\eta^\xi_t)_{t\ge 0}=(\eta^{\xi,N}_t)_{t\ge 0}$ denotes a trajectory of ${\rm Avg}(\T^d_N)$ when starting at time $t=0$ from $\xi\in \cP(\T^d_N)$,  $\pi=\pi_N$ is the uniform distribution on $\T^d_N$, and the above convergence holds a.s.\ with respect to the random sequence of updates.  For the probability law corresponding to such random updates, we write $\P=\P^N$, while $\E=\E^N$ stands for the corresponding expectation.

\subsection{Mixing for the averaging process}
The statement in \eqref{eq:conv-qualitative} concerns the qualitative long-run behavior of the averaging process on a fixed-size torus. A mixing analysis aims at quantifying such a convergence in the limit as $N\to \infty$, with the scope of   better capturing the relevant timescales governing the relaxation of  $\eta^\xi_t$, when starting from a worst-case initial condition $\xi\in \cP(\T^d_N)$.

As a sensible notion  of distance between  the law of $\eta_t^\xi$ and that of $\pi$,   we consider the  (mean) $L^p$-distance to equilibrium, given, for all $N\in \N$ and $p\in [1,2]$, by	
\begin{equation}
	\label{eq:wasserstein}	  \E\bigg[\bigg\|\frac{\eta_t^\xi}{\pi}-1\bigg\|_p^p\bigg]^{\frac1p}\eqdef \bigg(\sum_{x\in \T^d_N} \pi(x)\,\E\bigg[\bigg|\frac{\eta_t^\xi(x)}{\pi(x)}-1\bigg|^p\bigg]\bigg)^\frac1p\comma\qquad\text{with}\ \xi\in \cP(\T^d_N)\comma
\end{equation}
and define the corresponding (worst-case) $L^p$-mixing time as the first time $t\ge 0$ for which the above quantity falls below $1/2$ for all initial conditions $\xi\in \cP(\T^d_N)$. 
Note that 	  $\norm{\emparg}_p=\norm{\emparg}_{p,N}$ in \eqref{eq:wasserstein} stands for the usual $L^p$-norm on $(\T^d_N,\pi)$. Further, we remark that  \eqref{eq:wasserstein} may also be interpreted as the $L^p$-Wasserstein metrics between the random  $\eta_t^\xi \in \cP(\T^d_N)$ and the deterministic $\pi\in \cP(\T^d_N)$, when endowing $\cP(\T^d_N)$ with the distance induced by $\norm{\emparg}_p$.	Note that stronger distances,  e.g., total variation,  are less relevant in the context of the averaging process; indeed, the equilibrium $\pi$ is deterministic, while $\eta^\xi_t$ has a non-degenerate random  distribution at any time $t> 0$, provided that $\xi\neq \pi$. 

Within this setting, a worst-case mixing analysis for ${\rm Avg}(\T^d_N)$ has been carried out in \cite{quattropani2021mixing}. There, the authors show that mixing occurs gradually at times $t=\Theta(N^2)$. This result was actually derived for a larger class of graph sequences, referred to as \textquotedblleft finite-dimensional\textquotedblright\ \cite[Assumption 1]{quattropani2021mixing}, establishing in all these examples that mixing of the averaging process was dictated by that of the corresponding simple random walk on the same graph. Other settings for which sharp asymptotics for the $L^p$-mixing times, with $p=1,2$, have been established are  the hypercube \cite{caputo_quattropani_sau_cutoff_2023}, and  the complete and complete bipartite graphs \cite{chatterjee2020phase,caputo_quattropani_sau_cutoff_2023}. In these last three examples, cutoff occurs. For a quick overview on these results and on some general tools recently developed for the study of mixing times of the averaging process, we refer the interested reader to \cite[Section 2]{caputo_quattropani_sau_cutoff_2023}.

\subsection{Averaging process as a noisy heat flow}\label{sec:RW}
The averaging process may also be viewed as the probability distribution of a random walk $(U_t)_{t\ge 0}$ in a highly degenerate dynamic random environment, in which the time-dependent conductances between nearest neighbors are either zero, or  become \textquotedblleft instantaneously infinite\textquotedblright\ as soon as the Poisson clock associated to the pair  rings. At the occurrence of such an event, the walk originally sitting on one of the vertices will find itself  with equal probability on either one of the two nearest neighbors. 

With this interpretation, while a realization  of $\eta_t^\xi\in \cP(\T^d_N)$ coincides with the distribution of $U_t$ for a quenched realization of the random environment, the mean $\E[\eta^\xi_t]\in \cP(\T^d_N)$ shall be considered as the corresponding annealed law. As already observed in \cite[Lemma 1]{aldous_lecture_2012} and exploited in \cite{quattropani2021mixing,caputo_quattropani_sau_cutoff_2023} (see also \cite[Section 6.3]{tran2023cutoff} for a recent use of this fact),  $t\mapsto \E[\eta^\xi_t]$ describes the heat flow of ${\rm RW}(\T^d_N)$, the \textquotedblleft lazy\textquotedblright\ continuous-time simple random walk $(X_t)_{t\ge 0}$ on $\T^d_N$, with $X_0 \sim \xi$ and infinitesimal generator $L^{\rm RW}=L_N^{\rm RW}$ given, for all  $\psi\in \R^{\T^d_N}$, by
\begin{equation}\label{eq:generator-RW}
	L^{\rm RW} \psi(x)=\frac12 \sum_{\substack{y\in \T^d_N\\
			|x-y|=1}}\big(\psi(y)-\psi(x)\big)\comma\qquad  x\in \T^d_N\fstop
\end{equation}
Hence, letting $\Pr_\xi^{\rm RW}$ denote the law associated to this walk and  $\pi^\xi_t$ its distribution at time $t\ge 0$, we obtain
\begin{equation}\label{eq:mean}
	\E\big[\eta^\xi_t(x)\big]=\pi^\xi_t(x)\eqdef {\mathbf P}_\xi^{\rm RW}\big(X_t=x\big)\comma\qquad x\in \T^d_N\comma t \ge 0\comma \xi\in \cP(\T^d_N)\fstop
\end{equation} 
In other words, \eqref{eq:mean} states that the  \textquotedblleft noisy\textquotedblright\ jump process $\eta_t^\xi$ has the heat flow $\pi_t^\xi$ as its mean. Although the heat flow does not necessarily  identify the correct mixing behavior of the averaging process on \textit{all} graphs, we will prove that a strong form of concentration of $\eta_t^\xi $ around $\pi^\xi_t$ holds on	 $\T^d_N$.

\subsection{Summary of main results}\label{sec:summary-results}
Let us provide a preliminary informal description of the two main features of ${\rm Avg}$ that we investigate in this article, together with some of their applications. In what follows, the standard asymptotic notation \textquotedblleft $O(\emparg)$\textquotedblright, \textquotedblleft$o(\emparg)$\textquotedblright, \textquotedblleft$\Theta(\emparg)$\textquotedblright,  \textquotedblleft$\ll$\textquotedblright, etc., refers to the limit $N\to \infty$.

In what follows, we  focus 	 on worst-case initial conditions $\xi \in \cP(\T^d_N)$, which, by simple convexity arguments (see, e.g., \cite[Section 2.3]{caputo_quattropani_sau_cutoff_2023}), are all Dirac measures: for all $t\ge 0$,
\begin{equation}\label{eq:worst-case}
	\sup_{\xi\in \cP(\T^d_N)}\E\bigg[\bigg\|\frac{\eta_t^\xi}{\pi}-1\bigg\|_p^p\bigg]^\frac1p= \sup_{x\in \T^d_N} \E\bigg[\bigg\|\frac{\eta_t(x,\emparg)}{\pi}-1\bigg\|_p^p\bigg]^\frac1p\fstop
\end{equation}
(The  analogous identity for ${\rm RW}(\T^d_N)$ is well-known.)
Here and all throughout, we write $\eta_t(x,\emparg)=\eta_t^\xi$ and, similarly, $\pi_t(x,\emparg)=\pi_t^\xi$, whenever $\xi=\car_x$, for some $x\in \T^d_N$. Finally, in the translation-invariant context of the torus $\T^d_N$, $\E\big[\big\|\frac{\eta_t(x,\emparg)}{\pi}-1\big\|_p^p\big]$ and $\big\|\frac{\pi_t(x,\emparg)}{\pi}-1\big\|_p^p$ do not depend on $x\in \T^d_N$; hence, for notational convenience, we consider $\eta_t(0,\emparg)$ and $\pi_t(0,\emparg)$, i.e., we fix the location of the initial mass as  $x=0\in \T^d_N$ all throughout.

\subsubsection{Early concentration phenomenon} \label{sec:summary-concentration}

The  idea adopted in \cite{quattropani2021mixing} in order to identify the correct timescale at which the averaging process mixes is, in a nutshell, the following. First, identify the timescale at which ${\rm RW}$ mixes; there, compare the distance-to-equilibrium of ${\rm Avg}$ with the one of ${\rm RW}$, and prove that the resulting error becomes arbitrarily small on that same timescale.	 

More in detail, this strategy starts from two estimates. The first one --- the lower bound in \cite[Lemma 5.1]{quattropani2021mixing}, which  is an immediate consequence of Jensen inequality and \eqref{eq:mean} --- states that mixing for $\eta_t(0,\emparg)$ does not occur before that of $\pi_t(0,\emparg)$: for all $p\in [1,2]$  and $t\ge 0$,
\begin{align}\label{eq:general-lb}
	\bigg\|\frac{\pi_t(0,\emparg)}{\pi}-1\bigg\|_p\le 	\E\bigg[\bigg\|\frac{\eta_t(0,\emparg)}{\pi}-1\bigg\|_p^p\bigg]^{\frac1p}\fstop
\end{align}
The second one --- a slight sharpening of the upper bound in \cite[\S5.2]{quattropani2021mixing} 	 ---  follows by the triangle inequality for the $L^p$-Wasserstein metrics in \eqref{eq:wasserstein} and Jensen inequality: for all $p\in [1,2]$  and $t\ge 0$, 
\begin{align}\label{eq:general-ub}
	\E\bigg[\bigg\|\frac{\eta_t(0,\emparg)}{\pi}-1\bigg\|_p^p\bigg]^{\frac1p}
	&\le \bigg\|\frac{\pi_t(0,\emparg)}{\pi}-1\bigg\|_p + 
	\E\bigg[\bigg\|\frac{\eta_t(0,\emparg)}{\pi}-\frac{\pi_t(0,\emparg)}{\pi}\bigg\|_2^2\bigg]^{\frac12}\fstop
\end{align}
In view of these two bounds, if one expects the mixing of $\eta_t(0,\emparg)$ to be dictated by that of $\pi_t(0,\emparg)$, the key to get matching lower  and  upper bounds in \eqref{eq:general-lb}--\eqref{eq:general-ub} consists in efficiently estimating the second term on the right-hand side of \eqref{eq:general-ub}.  

By implementing this strategy for the specific example of $\T^d_N$, since $\pi_t(0,\emparg)$ mixes gradually at times $t=\Theta(N^2)$, showing
\begin{equation}\label{eq:vanishing-qualitative}
	\lim_{n\to \infty}	\E\bigg[\bigg\|\frac{\eta_{t_N}(0,\emparg)}{\pi}-\frac{\pi_{t_N}(0,\emparg)}{\pi}\bigg\|_2^2\bigg]=0\comma\qquad \text{for all}\ t_N\gg N^2\comma
\end{equation}
readily implies that also $\eta_t(0,\emparg)$ mixes at times $t=\Theta(N^2)$.
Proving \eqref{eq:vanishing-qualitative} is the main step in the proofs of \cite[Proposition 2.8 \& Theorem 2.9]{quattropani2021mixing}.

Our first main goal is the following improvement of \eqref{eq:vanishing-qualitative}:	
\begin{equation}\label{eq:concentration-summary}
	\lim_{n\to \infty} \E\bigg[\bigg\|\frac{\eta_{t_N}(0,\emparg)}{\pi}-\frac{\pi_{t_N}(0,\emparg)}{\pi}\bigg\|_2^2\bigg]=0\comma\qquad\text{for all}\  t_N\gg N^\frac{2d}{d+2}\fstop
\end{equation}
Since $N^{2d/(d+2)}\ll N^2$ for any $d\ge 1$, \eqref{eq:concentration-summary} captures a new timescale --- dimension-dependent, but always shorter than the diffusive one ---  after which $\eta_t(0,\emparg)$ and its expectation $\pi_t(0,\emparg)$ stay close to each other, although both of them are still far from stationarity.  For this reason, we refer to such a property of the averaging process as an \emph{early concentration phenomenon}. For the precise quantitative version of this result, see Theorem \ref{th:concentration} below.

In view of the bounds \eqref{eq:general-lb} and \eqref{eq:general-ub} --- which hold true on any graph --- early concentration of ${\rm Avg}$ may be considered as a  stronger form of mixing, whenever ${\rm Avg}$ and ${\rm RW}$ mix on the same timescale. On the one hand, this suggests that, just like the mixing behaviors of ${\rm Avg}$ and ${\rm RW}$ are comparable on a large class of geometries (see, e.g., \cite[Proposition 2.8]{quattropani2021mixing} and \cite[Theorem 1.1]{caputo_quattropani_sau_cutoff_2023}),  early concentration should not be an exclusive instance of the example of ${\rm Avg}(\T^d_N)$ considered here, and should be investigated in settings in which scaling limits of discrete gradients of the underlying heat flow are available. On the other hand, this automatically excludes early concentration from settings like the complete and complete bipartite graphs \cite{chatterjee2020phase,caputo_quattropani_sau_cutoff_2023}, in which mixing of ${\rm Avg}$ occurs on a strictly longer timescale than that of ${\rm RW}$. For what concerns $L^2$-mixing, this is further explained by the following identity, which resembles the inequality in \eqref{eq:general-ub} with $p=2$:
for all $t\ge 0$,
\begin{equation}\label{eq:identity-l2}
	\E\bigg[\bigg\|\frac{\eta_t(0,\emparg)}{\pi}-1 \bigg\|_2^2\bigg] = \bigg\|\frac{\pi_t(0,\emparg)}{\pi}-1\bigg\|_2^2 + \E\bigg[\bigg\|\frac{\eta_t(0,\emparg)}{\pi}-\frac{\pi_t(0,\emparg)}{\pi}\bigg\|_2^2\bigg]\fstop
\end{equation}
The  proof of \eqref{eq:identity-l2}  combines the definition of $L^2$-norm in \eqref{eq:wasserstein} and the identity in \eqref{eq:mean}.

A quantitative version of \eqref{eq:concentration-summary} bears several useful consequences related to both mixing and scaling limits of ${\rm Avg}(\T^d_N)$. 
For instance, as a refinement of mixing, 	 our estimate combined with a quantitative local CLT for ${\rm RW}(\T^d_N)$ identifies the \textit{limit profile} for ${\rm Avg}(\T^d_N)$: for all $p\in [0,1]$ and $t>0$,
\begin{equation}
	\E\bigg[\bigg\|\frac{\eta_{tN^2}(0,\emparg)}{\pi}-1\bigg\|_p^p\bigg]^\frac1p = \norm{h_t(0,\emparg)-1}_{L^p(\T^d)}+O\bigg(\frac1N\bigg)\comma
\end{equation}
where $h_t(0,\emparg)$ is the heat kernel of the diffusion on $\T^d$ with generator $\frac12\Delta$.  We detail this result and its simple proof in Proposition \ref{pr:limit-profile} below. 
For related recent results about limit profiles also in absence of cutoff, see, e.g., \cite{barrera2022gradual}.

Along the same lines, \eqref{eq:concentration-summary} with $t_N=tN^2$ may be also interpreted as a \emph{quantitative hydrodynamic limit} for ${\rm Avg}(\T^d_N)$ (Corollary \ref{cor:HDL}). Here,  the heat equation $\partial_t \rho=\frac12 \Delta \rho$ on $\T^d$ arises as a hydrodynamic equation, and convergence is	 established in the stronger $L^p$-Wasserstein metrics rather than in probability, as most typically done (see, e.g., the classical monograph \cite{kipnis_scaling_1999}).

\subsubsection{Fast local smoothness} Mixing of the averaging process typically deals with measuring how far the averaging process $\eta_t(0,\emparg)$ is from its global equilibrium $\pi$. Next to this measure of distance-to-equilibrium, Aldous and Lanoue \cite[Section 2.3]{aldous_lecture_2012} proposed a natural notion of \emph{local smoothness}, aiming at quantifying more accurately the local flatness of $\eta_t(0,\emparg)$. As intuitively clear, local smoothness should occur at least as fast as global mixing and, as we will show in Theorem \ref{th:concentration}, is intimately related in our context to the phenomenon of early concentration discussed in the previous paragraph. Nevertheless,  analyzing local smoothness in general settings turns out to be a delicate issue,   and, compared to the range of tools recently developed for the study of global mixing for ${\rm Avg}$, techniques and results for  local mixing are rather limited  (see, e.g., \cite[Proposition 4]{aldous_lecture_2012} and \cite[Corollary 1]{berthier_tight_2020}).

Following \cite{aldous_lecture_2012} and the analogy with the corresponding heat flow, we employ the (mean) Dirichlet form to quantify local roughness of the averaging process. Indeed, while  $L^p$-distances are natural quantities to measure  the  distance to equilibrium of $\pi_t(0,\emparg)$, the Dirichlet form associated to ${\rm RW}(\T^d_N)$ with generator $L^{\rm RW}$ \eqref{eq:generator-RW},  given by
\begin{equation}
	\cE(\psi)\eqdef \frac{1}{4N^d}\sum_{x\in \T^d_N}\sum_{\substack{y\in \T^d_N\\
			|x-y|=1}}\tonde{\psi(x)-\psi(y)}^2\comma\qquad \psi\in \R^{\T^d_N}\comma
\end{equation} 
relates to an infinitesimal, thus, local, stage of the convergence of $\pi_t(0,\emparg)$, as encoded in
\begin{equation}\label{eq:infinitesimal-RW}
	\frac{\rm d}{{\rm d}t}\bigg\|\frac{\pi_t(0,\emparg)}{\pi}-1\bigg\|_2^2 = -2\, \cE\bigg(\frac{\pi_t(0,\emparg)}{\pi}\bigg)\comma\qquad t \ge0\fstop
\end{equation}
In \cite[Proposition 2]{aldous_lecture_2012}, the authors prove an analogous identity for $\eta_t(0,\emparg)$:
\begin{equation}\label{eq:infinitesimal-AL}
	\frac{\dd}{\dd t}\E\bigg[\bigg\|\frac{\eta_t(0,\emparg)}{\pi}-1\bigg\|_2^2\bigg] = 	-\E\bigg[\cE\bigg(\frac{\eta_t(0,\emparg)}{\pi}\bigg)\bigg]\comma\qquad t\ge 0\fstop
\end{equation}
These two relations are, indeed, similar, although differing by a factor $2$; this dissimilarity is reflected by different exponential decay rates for the (mean) $L^2$-distances. 
Nevertheless, both \eqref{eq:infinitesimal-RW} and \eqref{eq:infinitesimal-AL}  show that quantifying the decay rates of the terms on the right-hand side	implies refined mixing results for $\pi_t(0,\emparg)$ and $\eta_t(0,\emparg)$, respectively.

A first simple estimate for the (mean) Dirichlet form along the averaging process may be obtained via a convexity argument as done to derive \eqref{eq:general-lb}. Indeed,  Jensen inequality and the identity \eqref{eq:mean} yield the following lower bound: 
\begin{align}\label{eq:dir-form-trivial-lb}
	\cE\bigg(\frac{\pi_t(0,\emparg)}{\pi}\bigg)\le \E\bigg[\cE\bigg(\frac{\eta_t(0,\emparg)}{\pi}\bigg)\bigg]\comma\qquad t \ge 0\fstop
\end{align}
Our main contribution is to show that, for ${\rm Avg}(\T^d_N)$, the inequality in \eqref{eq:dir-form-trivial-lb} can actually be reversed as follows
\begin{equation}\label{eq:dir-form-ub}
	\E\bigg[\cE\bigg(\frac{\eta_t(0,\emparg)}{\pi}\bigg)\bigg]\le C\, \phi(t/N^d)\, \cE\bigg(\frac{\pi_t(0,\emparg)}{\pi}\bigg)\comma \qquad t \ge 0\comma
\end{equation} at the negligible cost of including some dimension-dependent constant $C=C(d)>0$ and a term $\phi(t/N^d)$  which diverges only for very large 	times $t\gg N^d$.  Ultimately, since the Dirichlet form along the heat flow $\pi_t(0,\emparg)$ may be efficiently estimated on $\T^d_N$, this yields a quantitative control for the local smoothness of ${\rm Avg}(\T^d_N)$ \emph{for all times $t\ge 0$}. This allows, in particular, to capture both the fast local smoothening (on the same timescale at which early concentration in \eqref{eq:concentration-summary} occurs) of ${\rm Avg}(\T^d_N)$, as well as the correct exponential contraction rate for large times. As an immediate consequence of this, we obtain, by integrating over time this bound, the corresponding control for the global distance-to-equilibrium (Corollary \ref{cor:l2-bound}). The latter result and \eqref{eq:dir-form-ub} are both employed in Section \ref{sec:mixing-bin}.

The bound in \eqref{eq:dir-form-ub} and an improved version of \eqref{eq:dir-form-trivial-lb} are the content of Theorem \ref{th:local-conv}, and together represent the first sharp results concerning local smoothness for the averaging process. It is worth to emphasize that an analogous result, although for a different model, has been recently proven by Banerjee and Burdzy  in \cite[Theorem 2.6]{banerjee2020rates}. There, the authors study (among other things) the  \emph{smoothing process} on $\T^d_N$, a model discussed by  Liggett within the class of linear systems \cite[Chapter IX]{liggett_interacting_2005-1}. Like the averaging process, also the smoothing process falls into the class of Markovian mass redistribution models. However, since averages are performed over one vertex at the time, the smoothing dynamics does not conserve mass. This seemingly marginal variation leads to some qualitative differences between the two processes. For instance,  equilibrium for the smoothing process is a truly random mass profile, while this is not the case for the averaging. Moreover, as proved in \cite[Proposition 2.5]{banerjee2020rates}, averaging over a single vertex deterministically does not increase the value of $\cE(\frac{\eta_t}{\pi})$; as simple examples show, this monotonicity property  does not hold when performing our edge-averaging dynamics.  Nevertheless,   despite these dissimilarities, we	 import some of their arguments into our analysis, yielding comparable local smoothness behaviors of the two processes when considered on $\T^d_N$.

As already anticipated, fast local smoothness comes with some relevant applications. One of them, and probably the most important, is the early concentration phenomenon (Theorem \ref{th:concentration}). As a second application, we show  how to transfer the integral information encoded in the (mean) Dirichlet form into \textit{pointwise gradient estimates} for ${\rm Avg}(\T^d_N)$. This result is presented in Proposition \ref{pr:grad-estimates}, and may be viewed as an analogue of the annealed gradient estimates for random walks in dynamic random environment recently obtained in  \cite{deuschel2023gradient}.

%

\section{Main results}\label{sec:results}
We are now ready to present our main results and some of their consequences in full detail. Before that, recalling the definition of the random walk  ${\rm RW}(\T^d_N)$ (Section \ref{sec:RW}), we define $\lambda=\lambda^N>0$ as its \textit{spectral gap}, namely the smallest non-zero eigenvalue of  the (negative) generator $-L^{\rm RW}$ given in \eqref{eq:generator-RW}. It is well-known that $\lambda= 1-\cos(2\pi/N)=\frac{2\pi^2}{N^2}\tonde{1+O(\frac1{N^2})}$. Finally,  we write $t_{\rm rel}=t_{\rm rel}^N=1/\lambda=\frac{N^2}{2\pi^2}\tonde{1+O(\frac1{N^2	})}$ for the corresponding \textit{relaxation time}. 

In what follows,  $B, C, C_1, C_2, \ldots, C_1', C_2',\ldots > 0$ denote constants whose exact value is unimportant and may change from line to line. Moreover, unless stated otherwise, such constants may depend on $d\ge 1$, but not on other variables, e.g., $N\in \N$ and $t\ge 0$.  

\subsection{Local smoothness and gradients estimates}

The following is the first of our main results.

\begin{theorem}[Local smoothness]\label{th:local-conv}
	For all $d\ge1$, for all $N\in \N$ large enough, and for all $t\ge 0$, we have
	\begin{equation}\label{eq:local-conv}
		\left(1+C_1(t\wedge 1)\right)\cE\bigg(\frac{\pi_t(0,\emparg)}{\pi}\bigg)\le \E\bigg[\cE\bigg(\frac{\eta_t(0,\emparg)}{\pi}\bigg)\bigg]\le C_2 \exp\tonde{\frac{Bt}{N^{d+2}}} \cE\bigg(\frac{\pi_t(0,\emparg)}{\pi}\bigg)\comma
	\end{equation}
	for some constants $B, C_1, C_2> 0$ (depending only on $d\ge1$).
\end{theorem}
Although we formulated the above result as a comparison between the (mean) Dirichlet form of ${\rm Avg}(\T^d_N)$ and that of  ${\rm RW}(\T^d_N)$, the latter quantity is rather explicit thanks to the full knowledge of eigenvalues and eigenfunctions for ${\rm RW}(\T^d_N)$ (see also Section \ref{sec:proof-local-conv-renewal} below for more details). Indeed, one obtains, for all $N\in \N$ large enough and for all $t \ge 0$,
\begin{equation}\label{eq:estimate-dir-RW}
	C'_1\, \Xi(t)\le \cE\bigg(\frac{\pi_t(0,\emparg)}{\pi}\bigg)\le C'_2\, \Xi(t)\comma\qquad \text{with}\ \Xi(t)\eqdef \frac{N^d\exp\tonde{-2t/t_{\rm rel}}}{\tonde{N^{d+2}\wedge t^{d/2+1}}\vee 1}\comma
\end{equation}
for some constants $C'_1,C'_2>0$  (depending only on $d\ge1$). 

Plugging the bounds in \eqref{eq:estimate-dir-RW} into \eqref{eq:local-conv} allows to quantify the fast local smoothness of ${\rm Avg}(\T^d_N)$. Not only there follows that
\begin{equation}
	\E\bigg[\cE\bigg(\frac{\eta_{t_N}(0,\emparg)}{\pi}\bigg)\bigg] = \Theta\tonde{\frac1{N^2}}\comma\qquad \text{for all}\ t_N=\Theta(N^2)\comma
\end{equation} but we also extract quantitative information for $t\ll N^2$ and $t\gg N^2$. In particular, as long as $t=O(N^2)$, the resulting upper bound states that $\eta_t(0,\emparg)$ gets locally smoother at the polynomial scale $t^{-(d/2+1)}$  (just like $\pi_t(0,\emparg)$). For larger times, instead, local smoothness becomes exponential, with rate  $2/t_{\rm rel}\left(1+o(1)\right)$, thus, crucially recovering the factor $2$ in the exponential decay rate (cf.\ \eqref{eq:infinitesimal-RW} and \eqref{eq:infinitesimal-AL}). 
Let us further remark that, since $t_{\rm rel}=\frac{N^2}{2\pi^2}\tonde{1+o(1)}=o(N^{d+2})$, the exponential factor in the upper bound of \eqref{eq:local-conv} plays no significant role as long as $t=O(N^{d+2})$, a timescale much longer than the diffusive one. 
As for the lower bound in \eqref{eq:local-conv}, the factor $\tonde{1+C_1\tonde{t\wedge 1}}$ is  a seemingly irrelevant improvement upon \eqref{eq:dir-form-trivial-lb}. However, this term is strictly larger than one as soon as $t>0$, suggesting that, even  after a very small time, the fluctuations of the random gradients of $\eta_t(0,\emparg)$ become comparable with their mean.

Local smoothness is an integral (mean) quantity regarding the gradients of the averaging process. As a consequence of the estimate in Theorem \ref{th:local-conv}, we derive pointwise (rather than integral) estimates for such gradients. 
We remark that, keeping in mind the interpretation from Section \ref{sec:RW} of $\eta_t(x,y)$ as the distribution of a random walk in a  dynamic random environment, such bounds may be thought of as the analogues (for $p=2$) of the annealed gradient estimates recently obtained in \cite[Theorem 1.6(i)]{deuschel2023gradient} for the time-dependent random conductance model with space-time ergodic and uniformly elliptic environment on $\Z^d$.  We emphasize that, although our result resembles that in \cite{deuschel2023gradient}, the two proofs are quite different. 
We postpone the proof of the following result to Section \ref{sec:proofs-propositions}.
\begin{proposition}[Pointwise gradient estimates]\label{pr:grad-estimates}
	For all $d\ge 1$,  for all $N\in \N$ large enough, for all $t\ge 0$, and for all $x, y\in \T^d_N$ with $|x-y|=1$, we have 
	\begin{align}
		\E\big[\big(\eta_t(0,x)-\eta_t(0,y	)\big)^2\big] \le  \frac{C}{\tonde{N^{2d+2}\wedge t^{d+1}}\vee 1} \exp\tonde{-2t/t_{\rm rel}+	\frac{Bt}{N^{d+2}}}\comma
	\end{align}
	for some constants $B, C> 0$ (depending only on $d\ge 1$).
\end{proposition}

Let us conclude this discussion on local smoothness with an important result on mixing of ${\rm Avg}(\T^d_N)$, which will turn out useful in Section \ref{sec:mixing-bin}, and which is immediately derived by integrating the infinitesimal relation \eqref{eq:infinitesimal-AL} over the time interval $[t,+\infty)$, combined with the upper bounds in \eqref{eq:local-conv} and \eqref{eq:estimate-dir-RW}.  
\begin{corollary}\label{cor:l2-bound}
	For all $d\ge1$, for all $N\in \N$ large enough, and for all $t\ge 0$, we have
	\begin{equation}
		\E\bigg[\bigg\|\frac{\eta_t(0,\emparg)}{\pi}-1\bigg\|_2^2\bigg]\le \frac{CN^d}{\tonde{N^d\wedge t^{d/2}}\vee 1}\exp\tonde{-2t/t_{\rm rel}+\frac{Bt}{N^{d+2}}} \comma
	\end{equation} 
	for some constants $B, C>0$ (depending only on $d\ge1$).
\end{corollary}

\subsection{Concentration and limit profile}
We now present our second main result, whose proof is based on the following identity (which we prove in Lemma \ref{lemma:concentration1})
\begin{equation}
	\E\bigg[\bigg\|\frac{\eta_t(0,\emparg)}{\pi}-\frac{\pi_t(0,\emparg)}{\pi}\bigg\|_2^2\bigg] = \frac1{dN^d}\int_0^t \E\bigg[\cE\bigg(\frac{\eta_{t-s}(0,\emparg)}{\pi}\bigg)\bigg]\cE\bigg(\frac{\pi_s(0,\emparg)}{\pi}\bigg)\, \dd s\comma\qquad t\ge 0\comma
\end{equation}
and the findings in Theorem \ref{th:local-conv}.
\begin{theorem}[Concentration]\label{th:concentration}
	Recall $\Xi(t)$ from \eqref{eq:estimate-dir-RW}.
	For all $d\ge 1$, for all $N\in \N$ large enough,	 and for all $t\ge 0$, we have
	\begin{equation}\label{eq:concentration}
		C_1 \left(1\wedge t\right)\Xi(t)\le \E\bigg[\bigg\|\frac{\eta_t(0,\emparg)}{\pi}-\frac{\pi_t(0,\emparg)}{\pi}\bigg\|_2^2\bigg]\le C_2\left(1\wedge t\right)\Xi(t)\,\exp\left(\frac{Bt}{N^{d+2}}\right)\comma
	\end{equation}
	for some constants $B, C_1,C_2> 0$ (depending only on $d\ge 1$).
\end{theorem}

The proof of the following consequence of Theorem \ref{th:concentration} may be found in Section \ref{sec:proofs-propositions}.

\begin{proposition}[Limit profile]\label{pr:limit-profile}
	Fix $d\ge1$, and let $h_t(0,u)$, $u \in \T^d$ and $t\ge 0$, be the heat kernel of the diffusion on $\T^d$ with generator $\frac12 \Delta$, started from the origin $0\in \T^d$.
	
	Then, for all bounded intervals $[a,b]\subset (0,+\infty)$, there exists $C=C(a,b,d)>0$ satisfying, for all $p\in [1,2]$ and for all $N\in \N$ large enough,
	\begin{equation}
		\sup_{t\in [a,b]}\abs{\E\bigg[\bigg\|\frac{\eta_{tN^2}(0,\emparg)}{\pi}-1\bigg\|_p^p\bigg]^\frac1p-\big\|h_t(0,\emparg)-1\big\|_{L^p(\T^d)}}\le \frac{C}{N}\fstop
	\end{equation}
\end{proposition}

As already discussed in Section \ref{sec:summary-concentration},  specializing the upper bound in \eqref{eq:concentration} to times $t=\Theta(N^2)$ yields a quantitative hydrodynamic limit for ${\rm Avg}(\T^d_N)$. We report the precise statement below. Its simple proof is analogous to that of Proposition \ref{pr:limit-profile} and, thus, is left to the reader.
\begin{corollary}[Quantitative hydrodynamic limit]\label{cor:HDL}
	Fix $d\ge 1$ and, for all $g\in \cC(\T^d)$, let $(t,u)\in [0,\infty)\times \T^d\mapsto h_t^g(u)$ be the unique solution to the heat equation $\partial_t \rho=\frac12 \Delta \rho$ on $\T^d$ with initial condition $\rho_0=g$.

	Then, for all bounded intervals $[a,b]\subset (0,+\infty)$, there exists $C=C(a,b,d)>0$ satisfying,  for all $N\in \N$ large enough, for all $\xi \in \cP(\T^d_N)$, and  $\Psi\in \cC(\T^d_N)$, 
	\begin{equation}
		\sup_{t\in [a,b]}	\E\bigg[\bigg|\sum_{x\in \T^d_N}\eta^\xi_{tN^2}(x)\, \Psi(\tfrac{x}{N})-\int_{\T^d}h^g_t(u)\, \Psi(u)\, \dd u\bigg|\bigg]\le \|\Psi\|_\infty\bigg(\frac{C}{N}+\bigg\|\frac{\xi}{\pi}-\frac{g(\tfrac{\emparg}{N})}{\pi}\bigg\|_1\bigg)\fstop
	\end{equation}
\end{corollary}

\subsection{Structure of the paper} The rest of the paper is organized as follows. In Section \ref{sec:preliminaries}, we discuss some preliminary facts which allow us to conveniently reformulate our problem. In Sections \ref{sec:proof-local-conv} and \ref{sec:proof-concentration} we present the proofs of our main results, Theorems \ref{th:local-conv} and \ref{th:concentration}, respectively. The proofs of Propositions \ref{pr:grad-estimates} and \ref{pr:limit-profile} are the content of Section \ref{sec:proofs-propositions}.  As a further application of our main results, in Section \ref{sec:mixing-bin}, we prove cutoff (in total variation distance) for an interacting particle system dual to the averaging process, known as \emph{binomial splitting process}, strengthening a result in \cite{quattropani2021mixing}.

\section{Preliminaries}\label{sec:preliminaries}
In this section, we introduce  two auxiliary Markov processes, whose properties will play a key role in Sections \ref{sec:proof-local-conv} and \ref{sec:proof-concentration} below,  dedicated to the proofs of our main results.
\subsection{Coupled random walks}\label{sec:crw}

As discussed in \eqref{eq:mean}, first-order moments of the averaging process may be expressed in terms of the random walk $X_t$. This property is a particular instance of \emph{duality} (see, e.g., \cite{liggett_interacting_2005-1}). This connection was already noted in \cite{aldous_lecture_2012}, further exploited and generalized in  \cite{quattropani2021mixing,caputo_quattropani_sau_cutoff_2023}, and  holds more generally for all $k^{th}$-order moments. In particular, 	in order to prove our main results, we will need to introduce a dual system of $k=2$ interacting walks, which we now describe in detail.

Place two particles on the vertices of $\T^d_N$, and endow each unordered pair of nearest neighbors of $\T^d_N$ with exponential clocks of unit rate, independent over the pairs.  Assume that the clock associated to the pair $\{x,y\}$, $|x-y|=1$, rings. Particles sitting at that time on either $x$ or $y$ decide, independently from each other and with probability $1/2$, to change vertex, i.e., moving from $x$ to $y$ if originally at $x$, or vice versa.   Let $((X_t,Y_t))_{t\ge 0}\subset \T^d_N\times \T^d_N$ denote the  Markov process of the positions of these two particles. Both marginals $X_t$ and $Y_t$ evolve as two copies of ${\rm RW}(\T^d_N)$. Moreover, they move independently as long as they are at graph distance larger than one. However, when sitting at distance smaller than two, the two walks may interact by experiencing synchronous jumps. 	 For this reason, we refer to $(X_t,Y_t)$ as a system of two \emph{coupled random walks}, shortly ${\rm CRW}(\T^d_N)$. We let $\Pr_\nu^{\rm CRW}$ denote its law when $(X_0,Y_0)\sim \nu$.

The importance of ${\rm CRW}(\T^d_N)$ lies in the following second-order duality relation, analogous to that in \eqref{eq:mean}:
\begin{equation}\label{eq:dual-CRW}
	\E\big[\eta_t^\xi(x)\eta_t^\xi(y)\big] = \Pr_{\xi\otimes \xi}^{\rm CRW}(X_t=x\,,\,Y_t=y)\comma\qquad t\ge 0\comma x, y \in \T^d_N\comma \xi \in \cP(\T^d_N)\fstop
\end{equation}
In analogy with the notation used for ${\rm RW}(\T^d_N)$, when $\xi\otimes \xi'=\car_x\otimes \car_y$ for some $x, y	\in \T^d_N$, we simply write $\Pr_{x,y}^{\rm CRW}=\Pr_{\xi\otimes \xi'}^{\rm CRW}$.	
The introduction of ${\rm CRW}(\T^d_N)$ and the duality relation with ${\rm Avg}(\T^d_N)$ allow us to efficiently rewrite the mean Dirichlet form of ${\rm Avg}(\T^d_N)$.
\begin{proposition}\label{pr:dir-form-CRW}
	For all $d\ge 1$, $N\in \N$,  $t\ge 0$ and $e \in \T^d_N$ with $|e|=1$, we have
	\begin{align}\label{eq:dir-form-CRW}
		\frac1{dN^d}\,\E\bigg[\cE\bigg(\frac{\eta_t(0,\emparg)}{\pi}\bigg)\bigg]= \Pr_{0,0}^{\rm CRW}(X_t-Y_t=0)-\Pr_{0,0}^{\rm CRW}(X_t-Y_t=e) \fstop
	\end{align}
\end{proposition}
\begin{proof}
	By definition of $\cE(\emparg)$ and the duality relation \eqref{eq:dual-CRW}, we obtain
	\begin{align}
		&		\E\bigg[\cE\bigg(\frac{\eta_t(0,\emparg)}{\pi}\bigg)\bigg]
		= \frac{N^d}{4}\sum_{\substack{x,y\in \T^d_N\\
				|x-y|=1}}\E\big[\tonde{\eta_t(0,x)-\eta_t(0,y)}^2\big]\\
		&\qquad = \frac{N^d}{2}\sum_{x\in \T^d_N}\sum_{\substack{y\in \T^d_N\\
				|x-y|=1}}\set{ \E\big[\eta_t(0,x)^2\big]-\E\big[\eta_t(0,x)\eta_t(0,y)\big]}\\
		&\qquad= \frac{N^d}{2}\sum_{x\in \T^d_N}\sum_{\substack{y\in \T^d_N\\
				|x-y|=1}} \set{\Pr_{0,0}^{\rm CRW}(X_t=x\,,\,Y_t=x)- \Pr_{0,0}^{\rm CRW}(X_t=x\,,\,Y_t=y)}\fstop
	\end{align}
	We get the desired conclusion by rearranging the terms in the	 last summation as follows:
	\begin{align}
		\sum_{x\in \T^d_N}\sum_{\substack{y\in \T^d_N\\
				|x-y|=1}} \Pr_{0,0}^{\rm CRW}(X_t=x\,,\,Y_t=y)= 2d\, \Pr_{0,0}^{\rm CRW}(X_t=Y_t)\comma
	\end{align}
	\begin{align}
		\sum_{x\in \T^d_N}
		\sum_{\substack{y\in \T^d_N\\
				|x-y|=1}}\Pr_{0,0}^{\rm CRW}(X_t=x\,,\,Y_t=y)&= \sum_{\substack{e'\in \T^d_N\\
				|e'|=1}} \Pr_{0,0}^{\rm CRW}(Y_t=X_t+e')\\
		&= 2d\, \Pr_{0,0}^{\rm CRW}(X_t-Y_t=e)\comma
	\end{align}
	where for the last identity  we used the translation and permutation invariance of the marginals laws of $(X_t,Y_t)$, as well as  $(X_t,Y_t)\overset{\rm Law}= (Y_t,X_t)$ under $\Pr_{0,0}^{\rm CRW}$.
\end{proof}

\subsection{Difference process}\label{sec:difference-process}
Next to ${\rm CRW}(\T^d_N)$,  motivated by the result in Proposition \ref{pr:dir-form-CRW}, we consider a second auxiliary process.
Given  $\seq{(X_t,Y_t)}_{t\ge 0}$ evolving as  ${\rm CRW}(\T^d_N)$, define $(Z_t)_{t\ge 0}$ as	
\begin{equation}
	Z_t\eqdef X_t-Y_t\comma\qquad t \ge 0\fstop
\end{equation}
We refer to $Z_t$ as 	the \emph{difference process} associated to $(X_t,Y_t)$. In our context, $Z_t$ turns out to be a Markov process. Indeed, if we were considering the difference process $Z_t^0\eqdef X_t^0-Y_t^0$ associated to $(X_t^0,Y_t^0)$, two \emph{independent} copies of ${\rm RW}(\T^d_N)$, then it is well-known that $Z_t^0$ is the simple random walk on $\T^d_N$ with infinitesimal generator $A_0=2L^{\rm RW}$, i.e.,  $Z_t^0$ jumps like ${\rm RW}(\T^d_N)$, but at a double rate.  As for $Z_t$, due to the interaction between $X_t$ and $Y_t$ when $|X_t-Y_t|\le 1$,  $Z_t$ moves like $Z_t^0$ with a defect represented by 	slow bonds attached to the origin. More in detail, the infinitesimal generator $A$ of $Z_t$ reads as
\begin{align}
	A=A_0+R\comma
\end{align}
where we recall that $A_0=2L^{\rm RW}$ is the generator of $Z_t^0$, while, for all $\psi:\T^d_N\to \R$, 
\begin{align}
	R \psi(z)\eqdef \begin{dcases}
		-\tfrac12 A_0\psi(0) &\text{if}\ z=0\\
		-\tfrac12\tonde{\psi(0)-\psi(z)}+\tfrac14\tonde{\psi(-z)-\psi(z)} &\text{if}\ |z|=1\\
		0 &\text{else}\fstop  
	\end{dcases}
\end{align} 
Moreover, for all $t\ge 0$, let $S_t=e^{tA}$ (resp.\ $S_t^0=e^{tA_0}$) denote the transition kernels of the random walk $Z_t$ (resp.\ $Z_t^0$).	 Observe that, since both $A_0$ and $R$ are symmetric kernels,  $A$  and $S_t$ are symmetric, too. Symmetry of $S_t$ and  Proposition \ref{pr:dir-form-CRW} imply the following result.
\begin{proposition}
	For all $t\ge 0$, we have
	\begin{align}\label{eq:dir-form-S}
		\frac1{dN^d}\,\E\bigg[\cE\bigg(\frac{\eta_t(0,\emparg)}{\pi}\bigg)\bigg] =  S_t(0,0)-S_t(e,0)\fstop
	\end{align}
\end{proposition} 	
We conclude this section by observing that similar computations to those employed to derive \eqref{eq:dir-form-CRW} readily yield the following analogue of \eqref{eq:dir-form-S}: for all $t\ge 0$,
\begin{align}\label{eq:dir-form-S0}
	\frac1{dN^d}\, \cE\bigg(\frac{\pi_t(0,\emparg)}{\pi}\bigg) = S_t^0(0,0)-S_t^0(e,0)\fstop
\end{align}
Hence, by \eqref{eq:dir-form-S}, \eqref{eq:dir-form-S0}, and Jensen inequality, we obtain, for all $t\ge 0$,
\begin{equation}\label{eq:ineq-S-S0}
	\begin{aligned}
		S_t(0,0)-S_t(e,0)&=\frac1{dN^d}\,\E\bigg[\cE\bigg(\frac{\eta_t(0,\emparg)}{\pi}\bigg)\bigg]\\
		&\ge \frac1{dN^d}\, \cE\bigg(\frac{\pi_t(0,\emparg)}{\pi}\bigg) = S_t^0(0,0)-S_t^0(e,0)\fstop
	\end{aligned}
\end{equation}

\section{Proof of Theorem \ref{th:local-conv}}\label{sec:proof-local-conv}
This section is devoted to the proof of Theorem \ref{th:local-conv}, which we split into three parts. First, we exploit the representations in \eqref{eq:dir-form-S} and \eqref{eq:dir-form-S0} to express the mean Dirichlet form of ${\rm Avg}(\T^d_N)$ as an infinite series (Section \ref{sec:proof-local-conv-renewal}). Then, for this infinite series, we provide lower and upper bounds in Sections \ref{sec:proof-local-conv-lb} and \ref{sec:proof-local-conv-ub}, respectively. We gather all these facts together in Section \ref{sec:proof-local-conv-conclusion}.

\subsection{A renewal-type equation}\label{sec:proof-local-conv-renewal}
Our main goal in this section is to find a closed-form expression for the remainder in the inequality \eqref{eq:ineq-S-S0}. For this purpose, 
we employ the identities in \eqref{eq:dir-form-S} and \eqref{eq:dir-form-S0} involving  Dirichlet forms and the transition kernels $S_t$ and $S_t^0$, respectively, both introduced in Section \ref{sec:difference-process}. 
Let us further recall the definition of the associated generators $A=A_0+R$ and $A_0$ from Section \ref{sec:difference-process}.  

We start by applying the  integration by parts formula: for all $x\in \T^d_N$ and $t\ge0$, 
\begin{align}
	S_t(x,0)-S_t^0(x,0)&= -\int_0^t \sum_{y\in \T^d_N} S_{t-s}^0(x,y)\tonde{A_0-A}S_s\car_0(y)\, \dd s\\
	&= \int_0^t S_{t-s}^0(x,0)\, RS_s\car_0(0)\, \dd s + \int_0^t \sum_{\substack{y\in \T^d_N\\
			|y|=1}} S_{t-s}^0(x,y)\, RS_s\car_0(y)\, \dd s\\
	&=-\frac12\int_0^t S_{t-s}^0(x,0)\sum_{\substack{y\in \T^d_N\\
			|y|=1}}\tonde{S_s(y,0)-S_s(0,0)}\dd s\\
	&\qquad -\frac12 \int_0^t \sum_{\substack{y\in \T^d_N\\
			|y|=1}} S_{t-s}^0(x,y)\tonde{S_s(0,0)-S_s(y,0)}\dd s\\
	&= \frac12 \int_0^t \sum_{\substack{y\in \T^d_N\\|y|=1}} \tonde{S_s(0,0)-S_s(y,0)}\tonde{S_{t-s}^0(x,0)-S_{t-s}^0(x,y)}\dd s\fstop
\end{align}
Noting that $S_s(e,0)=S_s(e',0)$ for all $e, e'\in \T^d_N$ with $|e|=|e'|=1$, we further obtain
\begin{equation}
	S_t(x,0)-S_t^0(x,0) = \int_0^t \tonde{S_s(0,0)-S_s(e,0)}\sum_{\substack{y\in \T^d_N\\ |y|=1}}\frac12\tonde{S_{t-s}^0(x,0)-S_{t-s}^0(x,y)}\dd s\fstop
\end{equation}
By evaluating the above expression at $x=0$ and $x=e\in \T^d_N$, and  subtracting, we get
\begin{align}\label{eq:ibp-final}
	&	S_t(0,0)-S_t(e,0)= S_t^0(0,0)-S_t^0(e,0)\\
	&+\int_0^t \tonde{S_s(0,0)-S_s(e,0)} \sum_{\substack{y\in \T^d_N\\|y|=1}}\frac12\tonde{S_{t-s}^0(0,0)-S_{t-s}^0(0,y)-S_{t-s}^0(e,0)+S_{t-s}^0(e,y)}\dd s	\fstop
\end{align}
By exploiting the symmetry and translation invariance of the random walk $Z_t^0$, we may simplify the summation above as follows: for $e,e'\in \T^d_N$ with $|e|=|e'|=1$ and $e'\neq \pm e$,
\begin{align}
	&	\sum_{\substack{y\in \T^d_N\\|y|=1}}\frac12\tonde{S_t^0(0,0)-S_t^0(0,y)-S_t^0(e,0)+S_t^0(e,y)}
	\\ 
	&\qquad= \tonde{d+1/2} S_t^0(0,0)-2d\, S_t^0(e,0)+\tonde{d-1} S_t^0(e+e',0)+1/2\, S_t^0(2e,0)\fstop
\end{align}
Hence, adopting the shorthand notation
\begin{align}
	\label{eq:f}
	f(t)&\eqdef
	\tonde{d+1/2} S_t^0(0,0)-2d\, S_t^0(e,0)+\tonde{d-1} S_t^0(e+e',0)+1/2\, S_t^0(2e,0)\comma
\end{align}
\begin{align}\label{eq:g-u}
	g(t)\eqdef S_t^0(0,0)-S_t^0(e,0)\comma\qquad u(t)\eqdef S_t(0,0)-S_t(e,0)\comma
\end{align}
the identity in \eqref{eq:ibp-final} reads as the following renewal equation
\begin{equation}\label{eq:renewal-equation}
	u(t)= g(t)+\int_0^t u(s)\, f(t-s)\, \dd s\comma\qquad t \ge 0\comma
\end{equation}
for which, by iterating this integral relation, a solution may be expressed in terms of an infinite series expansion involving only 	the functions $f$ and $g$:
\begin{equation}\label{eq:infinite-series-repr-u}
	u=g +\sum_{k=1}^\infty (g\ast f^{\ast k})=\sum_{k=0}^\infty \tttonde{g\ast f^{\ast k}}\comma\qquad \text{with}\ f^{\ast 0}\eqdef \delta_0\comma f^{\ast k}\eqdef f \ast f^{\ast (k-1)}\fstop
\end{equation}
Here, the symbol $\ast$ denotes the usual convolution for functions defined on $\R$, although we will always apply it to functions which vanish on $(-\infty,0)$. Let us further remark that, with our notation,  $f^{\ast 1}=f$.
Moreover, note that both $f$ and $g$ are bounded and continuous. Therefore, in order to ensure that $u=\sum_k \tonde{g\ast f^{\ast k}}$ is the unique solution to \eqref{eq:renewal-equation}, it suffices to show that the functions $f$ and $g$ (and, thus, their iterated convolutions) are non-negative. We prove this property in Lemma \ref{lemma:f-g-positive} below. In the same lemma, we gather other simple properties of these functions to be employed later.

Instrumentally to the proof of this lemma, we remark that the functions $f$ and $g$ may be further simplified, since they only depend on transition probabilities of the simple random walk $Z_t^0$ (without defects) on $\T^d_N$ and $Z_t^0$ moves independently among each of the $d$ components. Indeed, letting $p_t(i)=p_t(i,0)$, for $i\in \T_N$ and $t\ge 0$, denote the transition probabilities to the origin of the one-dimensional simple random walk on $\T_N$ (i.e., $Z_t^0$ for $d=1$), we obtain, for all $d\ge 1$, $x=(i_1,\ldots, i_d)\in \T^d_N$ and $t\ge 0$, 
\begin{equation}\label{eq:tensorization}
	S_t^0(x,0)= \prod_{\ell=1}^d p_t(i_\ell)\fstop
\end{equation}
Henceforth, we have, for all $d\ge 1$,
\begin{equation}\label{eq:g-decomposition}
	g(t)=p_t(0)^{d-1}\tonde{p_t(0)-p_t(1)}\comma
\end{equation}
while, for $d=1$,
\begin{equation}\label{eq:f-decomposition-1D}
	f(t)=3/2\, p_t(0)-2\, p_t(1)+1/2\, p_t(2)\comma
\end{equation}
and, for $d\ge 2$,
\begin{equation}
	f(t)
	= p_t(0)^{d-2}\quadre{d\tonde{p_t(0)-p_t(1)}^2+1/2\tonde{p_t(0)^2-2\,p_t(1)^2+ p_t(0)p_t(2)}}	\fstop
	\label{eq:f-decomposition-2D}
\end{equation}
In what follows, we will repeatedly employ this convenient rewriting in combination with the explicit eigendecomposition (or, equivalently, the Laplace inversion formula) of $p_t(i)$  (see, e.g., \cite[Section 12.3.1]{levin2017markov}): for all $i\in \T_N$ and $t\ge 0$,
\begin{equation}\label{eq:eigendecomposition-1D}
	p_t(i)=\frac1N\sum_{j=0}^{N-1} \cos\tonde{2\pi ij/N}\, \exp\tonde{-2\lambda_j^N t}\comma\quad \text{with}\ \lambda_j^N\eqdef 1-\cos\tonde{2\pi j/N}\ge 0\fstop
\end{equation}
\begin{lemma}\label{lemma:f-g-positive}For every $d\ge 1$,  we have:
	\begin{enumerate}[(a)]
		\item \label{it:lemma-f-g-1} For all $t>0$, $f(t)$ and $g(t)$ are positive and	 uniformly (with respect to $N\in \N$) bounded away from zero.
		\item \label{it:lemma-f-g-2} For all $N\in \N$ large enough, $f$ and $g$ are decreasing.
		\item \label{it:lemma-f-g-4}
		For all $N\in \N$ large enough, $f\le \tonde{d+1/2}g$.
	\end{enumerate}
\end{lemma}
\begin{proof}
	By the inversion formula \eqref{eq:eigendecomposition-1D}, we obtain at once that
	\begin{equation}
		t\longmapsto p_t(0)\qquad
		\text{and}\qquad t\longmapsto p_t(0)-p_t(1)
	\end{equation}
	are  decreasing and, for each fixed $t\ge 0$, uniformly (in $N\in \N$) bounded away from zero.	A similar conclusion holds for 
	\begin{align}
		t\longmapsto 3/2\, p_t(0)-2\, p_t(1)+1/2\, p_t(2)\qquad \text{and}\qquad t\longmapsto p_t(0)^2-2\, p_t(1)^2+p_t(0)p_t(2)\comma
	\end{align}
	as a consequence of the following two inequalities, respectively:
	\begin{equation}
		3/2-2\, \cos a + 1/2\, \cos 2a = 4\tonde{\sin\tonde{a/2}}^4\ge 0\comma \qquad a\in \R\comma
	\end{equation}
	\begin{equation}
		1-2\, \cos a\cos b+1/2\, \cos 2a+1/2\, \cos 2b = \tonde{\cos a-\cos b}^2\ge 0\comma\qquad a, b \in \R\fstop	
	\end{equation}
	Therefore, the claims in items \eqref{it:lemma-f-g-1}--\eqref{it:lemma-f-g-2} hold true because $f$ and $g$ in \eqref{eq:g-decomposition}--\eqref{eq:f-decomposition-2D} are products and sums of functions satisfying such properties.
	
	As for item \eqref{it:lemma-f-g-4}, we  claim that $p_t(2)\le p_t(1)$ for all $t\ge 0$. {For this purpose, let us define $t\in [0,\infty)\mapsto q_t\eqdef p_t(1)-p_t(2)$, and observe that, for all $N\ge 4$, 
		$q_0=0$, whereas, since $p_t(1), p_t(2)\to \frac1N$ as $t\to \infty$,  $\lim_{t\to \infty}q_t=0$.
		In view of \eqref{eq:eigendecomposition-1D}, we have, for all $t\ge 0$,
		\begin{equation}\label{eq:q't}
			q_t' = \frac1N\sum_{j=0}^{N-1}\set{2\tonde{\cos\tonde{2\pi j/N}-\cos\tonde{4 \pi j/N}}\tonde{\cos\tonde{2 \pi j/N}-1}} \exp\tonde{-2 \lambda_j^N t}\comma
		\end{equation}from which we obtain $q_0'=1$. Indeed,  $\frac2N\sum_{j=0}^{N-1}\cos\tonde{2 \pi i j/N}\cos\tonde{2 \pi \ell j/N}=0$ if  $(i,\ell)=(0,1), (0,2), (1,2)$, whereas $=1$ if $(i,\ell)=(1,1)$. Furthermore,  $t\mapsto q_t'$ must become negative because $q_0=\lim_{t\to \infty}q_t=0$, and we now argue that it changes sign only once, thus, ensuring that $q_t\ge 0$ for all $t\ge 0$. The expression within curly brackets in \eqref{eq:q't} reads as
		\begin{equation}
			\kappa_j^N\eqdef-{16} \tonde{\sin\tonde{\pi j/N}}^4 \tonde{1/2+\cos\tonde{2 \pi j/N}}\comma
		\end{equation}
		which is non-negative if and only if $j \in [N/3,2N/3]$.  Moreover, since  $\kappa_j^N= \kappa_{N-j}^N$ and $\lambda_j^N = \lambda_{N-j}^N$, we get the following decomposition of $q_t'$ in terms of  non-negative and non-positive parts (in what follows, we omit to write the integer part of $N/3$, $N/2$, $2N/3$):
		\begin{equation}
			q_t'= r_t^-+r_t^+\eqdef \frac2N\sum_{j=1}^{ N/3} \kappa_j^N\exp\tonde{-2\lambda_j^N t} + \frac1N \sum_{j= N/3 +1}^{ 2N/3-1} \kappa_j^N\exp\tonde{-2\lambda_j^N t}\fstop
		\end{equation}  Let $T=T_N> 0$ be the first time $t> 0$ satisfying $q_t' =0$. Then, since $j\in \set{1,\ldots,  N-1 }\mapsto\lambda_j^N \in (0,2]$ is symmetric around $N/2$ and increasing on $\set{1,\ldots, N/2}$,  we get, for all $s\ge 0$, 
		\begin{equation}
			q_{T+s}'\le r_T^-\exp\tonde{-2 \lambda_{\lfloor N/3\rfloor}^Ns}  + r_T^+ \exp\tonde{-2 \lambda_{\lfloor N/3\rfloor+1}^Ns} \le q_T'\exp\tonde{-2\lambda_{\lfloor N/3\rfloor}^N s}\le 0\comma 
		\end{equation}
		where the last step used the definition of $T$. Altogether, since we showed that $q_0'=1>0$, this proves that, for all $N\ge 4$ and $t\ge 0$, $q_t\ge 0$,
		that is,
		\begin{equation}\label{eq:p21}
			p_t(2)\le p_t(1)\comma \qquad t \ge 0\fstop 
	\end{equation}}
	Finally, inserting \eqref{eq:p21} into the expressions for $f$ in \eqref{eq:f-decomposition-1D}--\eqref{eq:f-decomposition-2D}, we obtain, 
	for $d=1$,
	\begin{align}
		f(t)\le 3/2\tonde{p_t(0)-p_t(1)}=3/2\, g(t)\comma\qquad t \ge 0\comma
	\end{align}
	while,  for $d\ge 2$,	
	\begin{align}
		f(t)&\le  p_t(0)^{d-2}\quadre{d\tonde{p_t(0)-p_t(1)}^2+1/2\tonde{p_t(0)-p_t(1)}^2}\\
		&= \tonde{d+1/2}p_t(0)^{d-2}\tonde{p_t(0)-p_t(1)}^2\\
		&\le \tonde{d+1/2}p_t(0)^{d-1}\tonde{p_t(0)-p_t(1)} = \tonde{d+1/2}g(t)\comma\qquad t \ge 0\comma
	\end{align}
	where for the last inequality we used the crude bound $p_t(0)-p_t(1)\le p_t(0)$.
	This concludes the proof of the lemma.
\end{proof}

\subsection{Lower bound in Theorem \ref{th:local-conv}}\label{sec:proof-local-conv-lb}
{In view of \eqref{eq:dir-form-S}--\eqref{eq:dir-form-S0},  \eqref{eq:f}--\eqref{eq:g-u}, and the infinite series representation in \eqref{eq:infinite-series-repr-u},
	we recast the lower bound in \eqref{eq:local-conv} in terms of the functions $u(t)$ and $g(t)$. This is the content of the following lemma. Remark that, as in Theorem \ref{th:local-conv}, for $t> 0$, we aim at a strict, uniform in $N\in \N$, inequality $u\gneq g$, improving upon the obvious estimate $u\ge g$ given in \eqref{eq:ineq-S-S0}.}
\begin{lemma}\label{lemma:local-conv-lb}
	For all $d\ge 1$, there exists $c=c(d)\in (0,1)$ satisfying, for all $N\in \N$ large enough  and for all  $t\ge 0$,
	\begin{equation}\label{eq:u-g-lb}
		u(t)\ge \exp\tonde{c\tonde{1\wedge t}}g(t)\fstop
	\end{equation}
\end{lemma}
\begin{proof}
	We show that,  for all $d\ge 1$ and $N\in \N$, 
	\begin{align}\label{eq:claim-u-g-lb}
		u(t)\ge \tonde{\sup_{s\in [0,t]}\exp\tonde{sf(s)}} g(t)\comma\qquad t \ge 0\comma
	\end{align}
	holds true. This suffices since $f$ is decreasing, $f(0)=d+1/2$ and,  for each $s> 0$, $f(s)$ is uniformly (in $N\in \N$) bounded away from zero (Lemma \ref{lemma:f-g-positive}).	
	
	In what follows, we derive \eqref{eq:claim-u-g-lb} by showing that
	\begin{equation}\label{eq:fk}
		\int_0^t f^{\ast k}(s)\, \dd s \ge \frac{\tonde{tf(t)}^k}{k!}\comma\qquad t \ge 0\comma
	\end{equation}
	holds for all $k\in \N$. We show \eqref{eq:fk} by induction on $k\in \N$.
	The claim for $k=1$, namely $\int_0^t f(s)\, \dd s \ge tf(t)$,	 clearly holds since $f$ is decreasing. Fix $k\ge 2$, and assume  \eqref{eq:fk} to be true for   $k-1$; then, for all $t\ge 0$,
	\begin{align}
		\int_0^t f^{\ast k}(s)\, \dd s&=\int_0^t  \int_0^s f(s-r)\, f^{\ast (k-1)}(r)\, \dd r\,\dd s\\
		&\ge	\int_0^t  f(s)\int_0^sf^{\ast(k-1)}(r)\, \dd r\,\dd s\\
		&\ge\int_0^tf(s)^k\frac{s^{k-1}}{(k-1)!}\,\dd s\\
		&\ge f(t)^k\int_0^t \frac{s^{k-1}}{(k-1)!}\,\dd s = \frac{\tonde{tf(t)}^k}{k!}\comma
	\end{align}
	where for the first and third inequalities we used the monotonicity and positivity of $f$, while for the second one we used the induction hypothesis. This shows the validity of \eqref{eq:fk}.
	The desired conclusion now follows by the infinite-series representation of $u(t)$ in \eqref{eq:infinite-series-repr-u}, as well as the monotonicity of $g$ and the positivity of $f$ (Lemma \ref{lemma:f-g-positive}): for all $t\ge s \ge 0$,
	\begin{align}
		u(t)&=
		g(t)+\sum_{k=1}^\infty\int_0^t g(t-r)\, f^{\ast k}(r)\, \dd r\\
		&\ge g(t)+ g(t)\sum_{k=1}^\infty \int_0^s f^{\ast k}(r)\, \dd r\\
		&\ge g(t)\tonde{\sum_{k=0}^\infty \frac{\tonde{sf(s)}^k}{k!}} = g(t)\exp\tonde{sf(s)}\comma
	\end{align}
	where the last inequality is a consequence of \eqref{eq:fk} and positivity of $g$. Optimizing over $s\le t$ yields \eqref{eq:claim-u-g-lb} and, thus, concludes the proof.
\end{proof}
\subsection{Upper bound in Theorem \ref{th:local-conv}}
\label{sec:proof-local-conv-ub}
Also for the derivation of the upper bound in Theorem \ref{th:local-conv}, we exploit the infinite-series representation \eqref{eq:infinite-series-repr-u} for $u$ involving convolutions of the functions $f$ and $g$. However, since we are not interested in optimizing constants in this case, we may employ $f\le \tonde{d+1/2}g$ from Lemma \ref{lemma:f-g-positive}\eqref{it:lemma-f-g-4} and positivity of $f$ and $g$ (Lemma \ref{lemma:f-g-positive}\eqref{it:lemma-f-g-1}) to	 estimate $u(t)$ from above by $\tilde u(t)$, defined as
\begin{equation}\label{eq:u-tilde}
	\tilde u\eqdef
	\sum_{k=1}^\infty \tilde g^{\ast k}\comma\qquad \text{with}\ \tilde g\eqdef \tonde{d+1/2}g\fstop
\end{equation}
Hence, the following estimate combined with \eqref{eq:estimate-dir-RW} (cf.\ \eqref{eq:g-u} and \eqref{eq:dir-form-S0}) would conclude the proof of the upper bound in Theorem \ref{th:local-conv}.
\begin{lemma}\label{lemma:final-local-conv}
	For all $d\ge 1$ and for all $N\in \N$ large enough, we have
	\begin{equation}\label{eq:final-lemma-ub}
		\sum_{k=1}^\infty \tilde g^{\ast k}(t)\le \frac{C}{N^{d+2}\wedge t^{d/2+1}}\exp\tonde{-2t/t_{\rm rel}+\frac{Bt}{N^{d+2}}}\comma \qquad t> 0\comma
	\end{equation}
	for some constants $B,C>0$ (depending only on $d\ge1$).
\end{lemma}
Before entering the details of the proof of Lemma \ref{lemma:final-local-conv}, we remark that, compared to the bound
\begin{equation}\label{eq:ub-g-tilde}
	\tilde g(t)\le \frac{C'}{N^{d+2}\wedge t^{d/2+1}}\exp\tonde{-2t/t_{\rm rel}}\comma\qquad t \ge 0\comma
\end{equation}
for some $C'=C'(d)>0$, 	estimating the infinite series in \eqref{eq:final-lemma-ub} requires an extra factor $\exp\tonde{\frac{Bt}{N^{d+2}}}$. With this approach , this exponential factor is actually unavoidable, as we now explain. Recall the lower bound for $g$ (thus, for $\tilde g$, too) in \eqref{eq:estimate-dir-RW}, analogous to \eqref{eq:ub-g-tilde}: for some $C''=C''(d)>0$,
\begin{equation}
	\tilde g(t)\ge \frac{C''}{\tonde{N^{d+2}\wedge t^{d/2+1}}\vee 1}\exp\tonde{-2t/t_{\rm rel}}\comma\qquad t \ge 0\fstop
\end{equation}
Finally, letting $c=c(t)\eqdef \frac{C''}{N^{d+2}}\car_{[0,\infty)}(t)$, we get $\tilde g(t)\ge c\exp\tonde{-2t/t_{\rm rel}}$ and, thus,	
\begin{equation}
	\sum_{k=1}^\infty \tilde g^{\ast k}(t)\ge \exp\tonde{-2t/t_{\rm rel}}\sum_{k=1}^\infty c^{\ast k}(t)=\exp\tonde{-2t/t_{\rm rel}}\sum_{k=1}^\infty \frac{(ct)^k}{k!}= \exp\tonde{-2t/t_{\rm rel}+ct}\fstop
\end{equation}
\begin{proof}[Proof of Lemma \ref{lemma:final-local-conv}]
	Since we can follow closely the proof in \cite[Lemma 5.6]{banerjee2020rates}, our task  boils down to establishing the analogues of the two preliminary lemmas therein, namely \cite[Lemmas 5.4 \& 5.5]{banerjee2020rates}. 
	
	We start by showing that,  for all $N\in \N$ large enough, we have
	\begin{equation}\label{eq:integral-tilde-g}
		\int_0^\infty \tilde g(t)\, \dd t \le \frac{d+1/2}{2d}\comma
	\end{equation}
	from which we then derive, for some $B=B(d)>0$ and $\theta\in (0,1)$, 
	\begin{equation}\label{eq:integral-tilde-g2}
		\int_0^\infty \exp\tonde{2t/t_{\rm rel}-\frac{Bt}{N^{d+2}}}\tilde g(t)\, \dd t\le \theta\fstop
	\end{equation}
	
	In order to prove \eqref{eq:integral-tilde-g}, recall $\tilde g\eqdef \tonde{d+1/2}g$, as well as \eqref{eq:g-u} and \eqref{eq:dir-form-S0}. Henceforth, 
	\begin{equation}
		\int_0^\infty \tilde g(t)\, \dd t = \frac{d+1/2}{2dN^d}\int_0^\infty 2\cE\bigg(\frac{\pi_t(0,\emparg)}{\pi}\bigg)\dd t = \frac{d+1/2}{2dN^d}\norm{\frac{\car_0}{\pi}-1}_2^2\le \frac{d+1/2}{2d}\comma
	\end{equation} 
	where the last step follows from	 $\norm{\frac{\car_0}{\pi}-1}_2^2=N^d-1$.	This yields \eqref{eq:integral-tilde-g}.
	
	For what concerns \eqref{eq:integral-tilde-g2}, arguing similarly as in the proof of \cite[Lemma 5.5]{banerjee2020rates}, we split the integral in \eqref{eq:integral-tilde-g2} into three pieces, so to obtain
	\begin{equation}\label{eq:integral-tilde-g3}
		\int_0^\infty \exp\tonde{2t/t_{\rm rel}-\frac{Bt}{N^{d+2}}}\tilde g(t)\, \dd t =	 \int_0^{\eps N^2}\cdots + \int_{\eps N^2}^{N^2} \cdots  + \int_{N^2}^\infty\cdots\comma
	\end{equation}
	for some well-chosen $\eps \in (0,1)$.
	We first deal with the first integral on the right-hand side above. Since $t_{\rm rel}\ge \delta N^2$, for some $\delta>0$, for all $\gamma\in (0,1/3)$, there exists $\eps=\eps(\delta,\gamma) >0		$ small enough such that, for all $N \in \N$ large enough,
	\begin{equation}
		\exp\tonde{2\eps N^2/t_{\rm rel}}\le1+\gamma\fstop
	\end{equation}
	With such a choice of $\gamma,\eps>0$, \eqref{eq:integral-tilde-g} ensures that
	\begin{equation}\label{eq:integral1-tilde-g}
		\int_0^{\eps N^2} \exp\tonde{2t/t_{\rm rel}-\frac{Bt}{N^{d+2}}}\tilde g(t)\, \dd t\le \tonde{1+\gamma}\int_0^\infty \tilde g(t)\, \dd t\le \frac{\tonde{1+\gamma}\tonde{d+1/2}}{2d}<1\fstop
	\end{equation}
	Note the strict inequality above. Now, the second and third integrals on the right-hand side of \eqref{eq:integral-tilde-g3} may be turned as small as desired via the estimate in \eqref{eq:ub-g-tilde}, by choosing first $N\in \N$ large enough and then $B>0$ large enough, respectively. Indeed, 
	\begin{equation}\label{eq:integral2-tilde-g}
		\int_{\eps N^2}^{N^2} \exp\tonde{2t/t_{\rm rel}-\frac{Bt}{N^{d+2}}}\tilde g(t)\, \dd t\le C'\int_{\eps N^2}^{N^2}t^{-(d/2+1)}\, \dd t\le CN^{-d}\comma
	\end{equation}	 
	for some $C=C(d,\eps)>0$, and
	\begin{equation}\label{eq:integral3-tilde-g}
		\int_{N^2}^{\infty} \exp\tonde{2t/t_{\rm rel}-\frac{Bt}{N^{d+2}}}\tilde g(t)\, \dd t\le \frac{C'}{N^{d+2}}\int_{N^2}^\infty \exp\tonde{-\frac{Bt}{N^{d+2}}}\dd t = \frac{C'}{B}\fstop
	\end{equation}
	Hence, \eqref{eq:integral1-tilde-g}, \eqref{eq:integral2-tilde-g} and \eqref{eq:integral3-tilde-g} yield \eqref{eq:integral-tilde-g2} for some $\theta\in (\frac{\tonde{1+\gamma}\tonde{d+1/2}}{2d},1)$ by choosing $\gamma>0$ small enough, and, then, $N\in\N$ and $B>0$ large enough.

	From this stage, the rest of the proof follows as in \cite[Lemma 5.6]{banerjee2020rates}; we streamline the main arguments in \cite[pp.\ 1150--1153]{banerjee2020rates} for the reader's convenience, while implementing the small modifications required. Recall $\theta \in (0,1)$ from \eqref{eq:integral-tilde-g2}, and fix  $\sigma \in (0,1)$ such that $\theta/\sigma \in (0,1)$.	For such a $\sigma \in (0,1)$, we define the sequence 
	\begin{equation}
		a_m\eqdef \frac12 \frac{\sigma^{\frac{2\tonde{m-1}}{d+2}}}{\sum_{j=1}^\infty \sigma^{\frac{2\tonde{j-1}}{d+2}}}\comma\quad m\ge 1\comma
	\end{equation}
	which, by definition, satisfies
	\begin{equation}\label{eq:properties-am}
		\sum_{m=1}^\infty a_m=1/2\qquad\text{and}\qquad a^*\eqdef \prod_{m=1}^\infty \tonde{1-a_m}>0\fstop
	\end{equation} Next, by adopting the following shorthand notation for the integrand in \eqref{eq:integral-tilde-g2}, 
	\begin{equation}\label{eq:rhs}
		h(t)\eqdef \exp\tonde{2t/t_{\rm rel}-\frac{Bt}{N^{d+2}}}\tilde g(t)\le \frac{C'}{t^{d/2+1}\wedge N^{d+2}}\defeq \tilde h(t)\comma\qquad t > 0\comma
	\end{equation}
	we claim that the following holds true: for all $k\in \N$ and $t>0 $, 
	\begin{equation}\label{eq:induction-banerjee-burdzy}
		h^{\ast k}(t)\le 2\theta^k \tilde h\bigg(t-t\sum_{m=1}^ka_m\bigg) + \theta^{k-1}\sum_{m=1}^k \tilde h(ta^*a_m)\fstop
	\end{equation}
	We refer to \cite[Eq.\ (5.35) \& pp.\ 1150--1151]{banerjee2020rates}  for the proof by induction of this inequality. Here, we only remark that  $t\mapsto \tilde h(t)$ given in \eqref{eq:rhs} (and, thus, the right-hand side of \eqref{eq:induction-banerjee-burdzy}) is a non-increasing function; this fact,   \eqref{eq:properties-am}, and the inequalities in \eqref{eq:integral-tilde-g2} and \eqref{eq:rhs} are the only inputs  required for the proof of \eqref{eq:induction-banerjee-burdzy}.	
	
	By \eqref{eq:induction-banerjee-burdzy} and the monotonicity of $t\mapsto \tilde h(t)$, we obtain, for all $t>0$,
	\begin{align}
		\sum_{k=1}^\infty h^{\ast k}(t)&\le 2\sum_{k=1}^\infty \theta^k \tilde h\bigg(t-t\sum_{m=1}^k a_m\bigg)
		+ \sum_{k=1}^\infty\theta^{k-1}\sum_{m=1}^k \tilde h(ta^*a_m)\\
		&\le \frac{2\theta}{1-\theta}\,\tilde h(t/2) +\frac1{1-\theta} \sum_{m=1}^\infty \theta^{m-1}\, \tilde h(ta^*a_m)
		\fstop \end{align}
	Now, for $t\le 2N^2$, since $ta^*a_m\le ta_m\le N^2$, we further get, for some $C=C(d)>0$,
	\begin{align}
		\sum_{k=1}^\infty h^{\ast k}(t)&\le \frac{2\theta C}{1-\theta}\frac{1}{t^{d/2+1}} + \frac{C'}{1-\theta}\sum_{m=1}^\infty \frac{\theta^{m-1}}{\tonde{ta^*a_m}^{d/2+1}}\\
		&\le \frac{2\theta C}{1-\theta}\frac{1}{t^{d/2+1}}+\frac{C''}{1-\theta}\frac{1}{t^{d/2+1}}\sum_{m=1}^\infty \tonde{\frac{\theta}{\sigma}}^{m-1}\\
		&\le \frac{C}{t^{d/2+1}}\comma
	\end{align} 
	where for the last inequality we used that $\theta/\sigma \in (0,1)$. Analogously, we have
	\begin{equation}
		\sum_{k=1}^\infty \tilde h^{\ast k}(t)\le \frac{C}{N^{d+2}}\comma\qquad t >2N^2\fstop
	\end{equation}
	Since
	\begin{equation}\label{eq:series-h-tilde-g}
		\sum_{k=1}^\infty h^{\ast k}(t)= \exp\tonde{2t/t_{\rm rel}-\frac{Bt}{N^{d+2}}}\sum_{k=1}^\infty \tilde g^{\ast k}(t)\comma \qquad t\ge 0	\comma
	\end{equation}
	this concludes the proof of the lemma.
\end{proof}
\subsection{Conclusion of the proof of Theorem \ref{th:local-conv}}\label{sec:proof-local-conv-conclusion}
\begin{proof}[Proof of Theorem \ref{th:local-conv}] Recall  \eqref{eq:dir-form-S}, \eqref{eq:dir-form-S0}, and \eqref{eq:g-u}. Then, the lower bound in \eqref{eq:local-conv} is the content of Lemma \ref{lemma:local-conv-lb}. The upper bound for times $t\ge 1$ may be obtained from $u\le \tilde u$ (defined in \eqref{eq:u-tilde}),  Lemma \ref{lemma:final-local-conv}, and the first inequality in \eqref{eq:estimate-dir-RW}. The same upper bound, but for small times $t\in [0,1]$, follows by the very definition of $u$ in \eqref{eq:g-u} and
	\begin{equation}\label{eq:u-small-times}
		u(t)=S_t(0,0)-S_t(e,0)\le S_t(0,0)\le 1\fstop
	\end{equation}
	This concludes the proof of the theorem.	
\end{proof}

\section{Proof of Theorem \ref{th:concentration}}\label{sec:proof-concentration}
We divide the proof of Theorem \ref{th:concentration} on the phenomenon of early concentration  into two parts. First, we express the quantity of interest in terms of (mean) Dirichlet forms for ${\rm Avg}(\T^d_N)$ and ${\rm RW}(\T^d_N)$. This is carried out in Lemma \ref{lemma:concentration1},  by specializing some expressions obtained in \cite{quattropani2021mixing,caputo_quattropani_sau_cutoff_2023} to our setting. Finally, thanks to this rewriting,  we exploit the estimates (and the proof arguments) from Theorem \ref{th:local-conv} to conclude.
\begin{lemma}\label{lemma:concentration1} For all $d\ge 1$,  $N \in \N$ and $t\ge 0$, we have
	\begin{align}\label{eq:Nt}
		&	\E\bigg[\bigg\|\frac{\eta_t(0,\emparg)}{\pi}-\frac{\pi_t(0,\emparg)}{\pi}\bigg\|_2^2\bigg] = \frac1{dN^d}\int_0^t \E\bigg[\cE\bigg(\frac{\eta_{t-s}(0,\emparg)}{\pi}\bigg)\bigg]\cE\bigg(\frac{\pi_s(0,\emparg)}{\pi}\bigg)\, \dd s\\
		&\qquad = dN^d \int_0^t \tonde{S_{t-s}(0,0)-S_{t-s}(e,0)}\tonde{S_s^0(0,0)-S_s^0(e,0)}\dd s	\comma
		\label{eq:Nt3}
	\end{align}
	for some $e\in \T^d_N$ with $|e|=1$.
\end{lemma}
\begin{proof}
	Let $\cN_t$ denote the left-hand side of \eqref{eq:Nt}. Then, by the last two displays in \cite[Proposition 2.5]{caputo_quattropani_sau_cutoff_2023}, we have
	\begin{equation}\label{eq:Nt2}
		\cN_t= \frac{N^d}{2}\int_0^t \frac12 \sum_{x\in \T^d_N}\sum_{\substack{y\in \T^d_N\\
				|x-y|=1}}\tonde{\pi_s(0,x)-\pi_s(0,y)}^2 \Phi_{t-s}(x,y)\,\dd s\comma
	\end{equation}
	where $\Phi_t(x,y)$ is defined in terms of ${\rm CRW}(\T^d_N)$ (Section \ref{sec:crw}) as follows:
	\begin{equation}
		\Phi_t(x,y)\eqdef \frac12 \tonde{\Pr_{x,x}^{\rm CRW}\big(X_t=Y_t\big)+\Pr_{y,y}^{\rm CRW}\big(X_t=Y_t\big)-2 \Pr_{x,y}^{\rm CRW}\big(X_t=Y_t\big)}\fstop
	\end{equation}
	By translation invariance of the dynamics of ${\rm CRW}(\T^d_N)$ and recalling the definition of the difference process $Z_t$ and its kernel $S_t$ (Section \ref{sec:difference-process}), we obtain, for all nearest neighbor vertices $x, y\in \T^d_N$, 
	\begin{equation}
		\Phi_t(x,y)= \Pr_{0,0}^{\rm CRW}\big(X_t=Y_t\big)-\Pr_{0,e}^{\rm CRW}\big(X_t=Y_t\big)= S_t(0,0)-S_t(e,0)\comma
	\end{equation}
	where $e\in \T^d_N$ is any vertex satisfying $|e|=1$. Plugging this identity into \eqref{eq:Nt2}, we get
	\begin{align}
		\cN_t=\int_0^t \tonde{S_{t-s}(0,0)-S_{t-s}(e,0)} \cE\bigg(\frac{\pi_s(0,\emparg)}{\pi}\bigg)\, \dd s\fstop
	\end{align}
	The identities in \eqref{eq:Nt} and \eqref{eq:Nt3}   follow by \eqref{eq:dir-form-S} and \eqref{eq:dir-form-S0}, respectively.	
\end{proof}

\begin{proof}[Proof of Theorem \ref{th:concentration}]
	Let us adopt the notation from Section \ref{sec:proof-local-conv}. Recalling the definitions	 of $u(t)$ and $g(t)$ in \eqref{eq:g-u}, the identity in \eqref{eq:Nt3} reads as
	\begin{equation}
		\frac{1}{dN^d}\,		\E\bigg[\bigg\|\frac{\eta_t(0,\emparg)}{\pi}-\frac{\pi_t(0,\emparg)}{\pi}\bigg\|_2^2\bigg] = 	 (u\ast g)(t)  \comma\qquad t \ge 0\fstop
	\end{equation}
	For the upper bound in \eqref{eq:concentration},  since  $u\le \tilde u$ and $0\le g\le \tilde g$ (cf.\ \eqref{eq:infinite-series-repr-u} and \eqref{eq:u-tilde}), we obtain
	\begin{equation}
		u\ast g\le \tilde u\ast g= \sum_{k=1}^\infty (\tilde g^{\ast k}\ast g) \le \sum_{k=2}^\infty\tilde g^{\ast k} \le \sum_{k=1}^\infty \tilde g^{\ast k} =\tilde u \comma
	\end{equation}
	and observe that the right-hand side was estimated, for all 	$t>0$ and $N\in \N$ large enough, in Lemma \ref{lemma:final-local-conv}. This provides the desired estimate for $t\ge1$; for $t\in [0,1]$, it suffices to recall  $g\le u\le 1$ (cf.\ \eqref{eq:infinite-series-repr-u} and \eqref{eq:u-small-times}), yielding
	\begin{equation}
		(u\ast g)(t)\le t\comma\qquad t\ge 0\fstop
	\end{equation}  For the lower bound in \eqref{eq:concentration}, using again $0\le g\le u$ and the monotonicity of $t\mapsto g(t)$ (Lemma \ref{lemma:f-g-positive}\eqref{it:lemma-f-g-2}), we get, for all $t\ge 0$,
	\begin{equation}
		(u\ast g)(t)\ge (g\ast g)(t)=2\int_0^{t/2} g(t-s)\, g(s)\, \dd s \ge 2g(t) \int_0^{{t/2}\wedge 1}g(s)\, \dd s\ge c\tonde{t\wedge 1}g(t)\comma
	\end{equation}	 
	where the last inequality follows  from Lemma \ref{lemma:f-g-positive}\eqref{it:lemma-f-g-1}--\eqref{it:lemma-f-g-2}. The first inequality in \eqref{eq:estimate-dir-RW} and \eqref{eq:dir-form-S0}, \eqref{eq:g-u} yield the desired lower bound, thus, concluding the proof of the theorem.
\end{proof}

\section{Proofs of Propositions \ref{pr:grad-estimates} and \ref{pr:limit-profile}}\label{sec:proofs-propositions}
We start with the proof of Proposition \ref{pr:grad-estimates}.

\begin{proof}[Proof of Proposition \ref{pr:grad-estimates}]
	Recall that, for $\P$-a.e.\ realization of the Poisson point process used for the updates, $\eta_t(0,\emparg)$ may also be interpreted as the probability distribution over $\T^d_N$ of the \textquotedblleft infinitesimal chunk\textquotedblright\ of mass $U_{0,t}\in \T^d_N$, which is a time-inhomogeneous random walk,  started in $0\in \T^d_N$, and evolving at later times as follows: nothing happens, until the vertex on which it sits, say $x\in \T^d_N$, experiences an update  with a nearest neighbor, say $y\in \T^d_N$; in that case, the chunk moves from $x$ to $y$ with probability $1/2$, while with the remaining probability it stays put. Since we are describing a time-inhomogeneous Markov process,  we adopt, only in this proof, the slightly more convenient notation $\eta_{s,t}(\emparg,\emparg)$, $s\le t$, indicating both starting and terminal times.	
	
	In view of this representation, Chapman-Kolmogorov formula holds in our context, and 	 reads as follows: for all $0\le s \le t$,
	\begin{equation}\label{eq:chapman-kolmogorov}
		\eta_{0,t}(0,x)-\eta_{0,t}(0,y) = \sum_{z\in \T^d_N}\eta_{0,s}(0,z)\tonde{\eta_{s,t}(z,x)-\eta_{s,t}(z,y)}\fstop
	\end{equation}
	Then, 
	by  \eqref{eq:chapman-kolmogorov} and Cauchy-Schwarz inequality, we get, for all $0\le s\le t$,
	\begin{align}
		\E\big[\big(\eta_t(0,x)-\eta_t(0,y)\big)^2\big]&= 	\E\bigg[\bigg(\sum_{z\in \T^d_N}\eta_{0,s}(0,z)\tonde{\eta_{s,t}(z,x)-\eta_{s,t}(z,y)}\bigg)^2\bigg]\\
		&\le \sum_{z\in \T^d_N}\E\big[\eta_{0,s}(0,z)\,	\big(\eta_{s,t}(z,x)-\eta_{s,t}(z,y)\big)^2\big]\\
		&=\sum_{z\in \T^d_N}\E\big[\eta_s(0,z)\big]\,\E\big[\big(\eta_{t-s}(z,x)-\eta_{t-s}(z,y)\big)^2\big]\\
		&= \sum_{z\in \T^d_N}\pi_s(0,z)\,\E\big[\big(\eta_{t-s}(z,x)-\eta_{t-s}(z,y)\big)^2\big]\comma
	\end{align}
	where the third step is a consequence of the fact that the Poisson updates over the time intervals $(0,s)$ and $(s,+\infty)$ are independent, while for the fourth one we used \eqref{eq:mean}. By using the well-known heat kernel estimate $\pi_s(0,z)\le C\tonde{\tonde{s^{d/2}\wedge N^d}\vee 1}^{-1}$, we further get
	\begin{align}
		\E\big[\big(\eta_t(0,x)-\eta_t(0,y)\big)^2\big]&\le \frac{C}{\tonde{s^{d/2}\wedge N^d}\vee 1}\sum_{z\in \T^d_N}\E\big[\big(\eta_{t-s}(z,x)-\eta_{t-s}(z,y)\big)^2\big]\\
		&= \frac{C}{\tonde{s^{d/2}\wedge N^d}\vee 1}\frac{2}{dN^d}\,\E\bigg[\cE\bigg(\frac{\eta_{t-s}(0,\emparg)}{\pi}\bigg)\bigg]\comma
	\end{align}
	where for the last step we used the duality relation \eqref{eq:dual-CRW}, the reversibility of ${\rm CRW}(\T^d_N)$ with respect to the counting measure on $\T^d_N\times \T^d_N$,  and Proposition \ref{pr:dir-form-CRW}. After inserting $s=t/2\wedge N^2$ in this last estimate, the upper bounds in \eqref{eq:local-conv} and \eqref{eq:estimate-dir-RW} yield the desired result.	
\end{proof}

We then conclude this section with the proof of Proposition \ref{pr:limit-profile} about the  limit profile.

\begin{proof}[Proof of Proposition \ref{pr:limit-profile}]
	Recall that, although we omit $N\in \N$ from the notation, $\norm{\emparg}_p$ denotes the $L^p$-norm on $(\T^d_N,\pi)$.
	By the triangle inequality and \eqref{eq:general-lb}--\eqref{eq:general-ub}, we get
	\begin{align}
		&\abs{\E\bigg[\bigg\|\frac{\eta_{tN^2}(0,\emparg)}{\pi}-1\bigg\|_p^p\bigg]^\frac1p-\big\|h_t-1\big\|_{L^p(\T^d)}}\\
		&\qquad\le {\E\bigg[\bigg\|\frac{\eta_{tN^2}(0,\emparg)}{\pi}-1\bigg\|_p^p\bigg]^\frac1p-\bigg\|\frac{\pi_{tN^2}(0,\emparg)}{\pi}-1\bigg\|_p}	\\
		&\qquad\quad +\abs{\bigg\|\frac{\pi_{tN^2}(0,\emparg)}{\pi}-1\bigg\|_p-\big\|h_t(\emparg/N)-1\big\|_p} + \abs{\big\|h_t(\emparg/N)-1\big\|_p-\big\|h_t-1\big\|_{L^p(\T^d)}}\\
		&\qquad \le \E\bigg[\bigg\|\frac{\eta_{tN^2}(0,\emparg)}{\pi}-\frac{\pi_{tN^2}(0,\emparg)}{\pi}\bigg\|_2^2\bigg]^\frac12\\
		&\qquad \quad +\bigg\|\frac{\pi_{tN^2}(0,\emparg)}{\pi}-h_t(\emparg/N)\bigg\|_p + \frac1N\, \sum_{i=1}^d\norm{\partial_i h_t}_{\cC(\T^d)}\comma
	\end{align}
	where the last term has been estimated via a first-order Taylor expansion (recall that $h_t\in \cC^\infty(\T^d)$, for all $t> 0$). The first term on the right-hand side above is $O(\frac1N)$ thanks to Theorem \ref{th:concentration}, while the second term is $o(\frac1N)$ by the quantitative local CLT for ${\rm RW}(\T^d_N)$ (see, e.g., \cite{lawler_limic_random_2010} for the analogous result on $\Z^d$). Since all these estimates are uniform over finite time intervals bounded away from zero and infinity, this concludes the proof.
\end{proof}

\section{An application to  cutoff of the  dual particle system}\label{sec:mixing-bin}

As a further consequence of Corollary \ref{cor:l2-bound} and the quantitative concentration result in Theorem \ref{th:concentration}, we provide sharp estimates for the mixing time of the particle system discussed in \cite{quattropani2021mixing}, dual to the averaging process, known as \emph{binomial splitting process}. In a few words, this model is the $k$-particle generalization  of the processes ${\rm RW}(\T^d_N)$ and ${\rm CRW}(\T^d_N)$ previously introduced. Here, we provide a quick description of this particle system and refer to \cite[Section 2.1]{quattropani2021mixing} (see also \cite[Section 1.2]{pymar2023mixing}) for more details, results, and background.

Given the discrete torus $\T^d_N$, $d\ge 1$, and an integer $k\in \N$, let $\Omega_k=\Omega_{N,k}$ denote the configuration space of $k$ indistinguishable particles on $\T^d_N$, namely,
\begin{equation}
	\Omega_k\eqdef \bigg\{\zeta \in \N_0^{\T^d_N}\bigg| \sum_{x\in \T^d_N}\zeta(x)=k\bigg\}\fstop
\end{equation}
On such a configuration space, we consider $(\zeta_t^k)_{t\ge 0}=(\zeta_t^{N,k})_{t\ge 0}$  as the Markov process which evolves as follows: start with a particle configuration $\zeta\in \Omega_k$ at time $t=0$; at the Poisson event times of ${\rm Avg}(\T^d_N)$ involving, say, the nearest neighbors $x, y \in \T^d_N$, reassign, independently and uniformly at random,  $x$ or $y$ as the new position of each particle which was originally sitting  on either $x$ or  $y$. The name \textquotedblleft binomial splitting\textquotedblright\ is now explained: the update  corresponds to placing on $x$ a  binomially-distributed fraction of the total $\zeta(x)+\zeta(y)$ particles. Further, the dynamics conserves the total number of particles, and, although particles interact when performing simultaneous jumps, at equilibrium, they are distributed as if they were independent. Indeed, a simple detailed balance computation ensures that 
\begin{equation}
	\mu_{k,\pi}= {\rm Multinomial}(k,\pi)\comma
\end{equation}
is the unique reversible measure for the particle system with $k$ particles. Finally, it is immediate to see that the particle systems with $k=1$ and $k=2$ particles correspond, respectively, to the processes ${\rm RW}(\T^d_N)$ and ${\rm CRW}(\T^d_N)$ described in Sections \ref{sec:RW} and \ref{sec:crw},  once the particles' labels are removed. 

We are interested in proving a total variation (TV) cutoff phenomenon for this process, as $k=k_N\to \infty$, strengthening the result in \cite[Theorem 2.3]{quattropani2021mixing}.
For this purpose, let, for all $k\in \N$,
$(P^k_t)_{t\ge 0}$ denote the Markov transition kernel associated to $(\zeta_t^k)_{t\ge 0}$, and write $\mu P^k_t$ for the distribution of $\zeta_t^k$ when  $\zeta_0^k\sim \mu$, for some $\mu \in \cP(\Omega_k)$. (Here, in analogy with $\cP(\T^d_N)$, $\cP(\Omega_k)$ stands for the space of probability distributions on $\Omega_k$.)  Further, we consider
\begin{equation}
	{\rm d}_k(t)={\rm d}_{N,k}(t)\eqdef 	\sup_{\mu \in \cP(\Omega_k)}\big\|\mu P_t^k-\mu_{k,\pi}\big\|_{\rm TV	}\comma\qquad t \ge 0\comma
\end{equation}
to encode	 the worst-case ${\rm TV}$-distance to equilibrium of the process as a function of time.
Then, recalling that $t_{\rm rel}=t_{\rm rel}^N>0$ denotes the spectral gap of ${\rm RW}(\T^d_N)$ (which, by \cite[Theorem 2.1]{quattropani2021mixing}, coincides with the spectral gap of the $k$-particle binomial splitting process), and letting	
\begin{equation}\label{eq:T(a)}
	T(a)=T_{N,k}(a)\eqdef \frac{t_{\rm rel}}{2}\tonde{\log k+aw_k}\comma\qquad a \in \R\comma
\end{equation}
the result in
\cite[Theorem 2.3]{quattropani2021mixing} specialized to the setting of the torus $\T^d_N$ reads as follows: 
\begin{equation}\label{eq:cutoff-bin-lb}
	\lim_{a\to -\infty} \liminf_{N\to \infty} {\rm d}_k(T(a))=1\comma
\end{equation}
\begin{equation}\label{eq:cutoff-bin-ub} \lim_{a\to \infty}\limsup_{N\to \infty} {\rm d}_k(T(a))=0\comma
\end{equation}
for $w_k\equiv 1$ and for 	  all  $k=k_N\to \infty$ satisfying $k=O(N^{2d})$. Our main improvement consists in extending this result from quadratic  to exponentially-many particles, paying the price of enlarging the so-called cutoff window $w_k$.
\begin{proposition}[Cutoff for binomial splitting process]\label{pr:cutoff-bin}
	Letting $w_k\eqdef 1\vee  \frac{\log k}{N^d}$, we have:
	\begin{enumerate}[(a)]
		\item \label{it:cutoff-lb}the claim in \eqref{eq:cutoff-bin-lb} holds for all $k\to \infty$ satisfying $\log k=o(N^d\log N)$;
		\item \label{it:cutoff-ub} the claim in \eqref{eq:cutoff-bin-ub} holds for all $k\to \infty$.
	\end{enumerate} 
	In particular, ${\rm TV}$-cutoff for $(\zeta_t^k)_{t\ge 0}$ holds for $k\to \infty$ satisfying $\log k=o(N^d\log N)$.
\end{proposition} 
We anticipate that, on the one hand,   the strategy in \cite{quattropani2021mixing} and our estimates on the fast local smoothness of ${\rm Avg}(\T^d_N)$ suffice to yield the desired upper bound, i.e., Proposition \ref{pr:cutoff-bin}\eqref{it:cutoff-ub}. On the other hand, for the lower bound, our results on early concentration of ${\rm Avg}(\T^d_N)$ combined with  the proof in \cite{quattropani2021mixing} fail to provide sharp estimates in the regime $\log k\ge N^d$. With the purpose of covering also this case, here, we follow a different proof strategy.

Finally, the key observation of the proof of Proposition \ref{pr:cutoff-bin} is that $(\zeta_t^k)_{t\ge 0}$ and ${\rm Avg}(\T^d_N)$ are related to each other via the following \textit{intertwining relation} \cite[Proposition 3.1]{quattropani2021mixing}:  for all $\xi\in \cP(\T^d_N)$, $k\in \N$, and $t\ge 0$, 
\begin{equation}\label{eq:intertwining-rel}
	\mu_{k,\xi}P_t^k = \E\big[\mu_{k,\eta_t^\xi}\big]\comma\qquad \text{with}\  \mu_{k,\eta}\eqdef {\rm Multinomial}(k,\eta)\comma \text{for all}\ \eta\in \cP(\T^d_N)\fstop
\end{equation}
We crucially exploit this fact for establishing both claims in Proposition \ref{pr:cutoff-bin}.

\begin{proof}[Proof of Proposition \ref{pr:cutoff-bin}]
	We divide the proof into two parts, starting with item \eqref{it:cutoff-ub}.

	\smallskip \noindent
	\emph{Upper bound \eqref{it:cutoff-ub}.}
	By (a slight generalization of) the intertwining  relation in \eqref{eq:intertwining-rel}, we have the following estimate
	\cite[Lemma 6.4]{quattropani2021mixing}: for all $k\in \N$ and $t\ge 0$, 
	\begin{equation}
		{\rm d}_k(t)\eqdef 	\sup_{\mu\in \cP(\Omega_k)}\big\|\mu P_t^k-\mu_{k,\pi}\big\|_{\rm TV}\le \sqrt{ek}\, \E\bigg[\bigg\|\frac{\eta_t(0,\emparg)}{\pi}-1\bigg\|_2^2\bigg]^\frac12\fstop
	\end{equation}
	Thanks to Corollary \ref{cor:l2-bound}, we further get, for some $C=C(d)>0$ and for all $t\ge N^2$, 
	\begin{align}
		{\rm d}_k(t)&\le  \frac{C\sqrt{k}\,N^{d/2}}{\tonde{N^{d/2}\wedge t^{d/4}}\vee 1}\exp\tonde{-t/t_{\rm rel}+\frac{Bt}{2N^{d+2}}}=  C\sqrt k\, \exp\tonde{-t/t_{\rm rel}+\frac{Bt}{2N^{d+2}}}\comma
	\end{align}
	where  for the last step we used $t\ge N^2$. Recall $t_{\rm rel}=\Theta(N^2)$, $k\to \infty$, and $w_k\eqdef 1\vee \frac{\log k}{N^d}$. Hence, by substituting $t=T(a)\gg N^2$ given in \eqref{eq:T(a)},   the right-hand side above reads as
	\begin{equation}
		C \exp\tonde{-\frac{a}2w_k + \frac{Bt_{\rm rel}}{4N^{d+2}}\tonde{\log k+a w_k}}\comma
	\end{equation}
	which vanishes taking first $N\to \infty$ and then $a\to \infty$.  This concludes the proof of the upper bound.

	\smallskip \noindent
	\emph{Lower bound \eqref{it:cutoff-lb}.}
	The approach we follow goes via the intertwining relation \eqref{eq:intertwining-rel}, which we employ as follows: 
	\begin{align}
		\label{eq:mixing-bin-lb1}
		\begin{aligned}
			{\rm d}_k(t)\ge	\big\|\mu_{k,\xi}P_t^k-\mu_{k,	\pi}\big\|_{\rm TV}& = \sup_{A\subset \Omega_k}\big(\E\big[\mu_{k,\eta_t^\xi}(A)\big]-\mu_{k,\pi}(A)\big)\\
			&\ge
			\E\big[\mu_{k,\eta_t^\xi}(A_*)\big]-\mu_{k,\pi}(A_*)\comma
		\end{aligned}
	\end{align}
	for some  $A_*\subset \Omega_k$ and $\xi\in \cP(\T^d_N)$ that we now specify.

	First, we set $\xi=\car_0$ and $t=T(a)$, for some $a<0$.
	With this choice, we have  (see, e.g., \cite[Lemma 20.11]{levin2017markov})
	\begin{align}
		\norm{\pi_t(0,\emparg)-\pi}_{\rm TV}\eqdef \sup_{B\subset \T^d_N}\big(\pi_t(0,B)-\pi(B)\big)\ge \frac{e^{-t/t_{\rm rel}}}2= 	\frac{e^{-aw_k/2}}{2\sqrt k}\fstop
	\end{align}
	Let $ B_*=B_*(N,a)\subset \T^d_N$ be a proper subset which attains the above supremum; hence,  \begin{equation}\label{eq:pit-pi-b}
		\pi_t(0,B_*)\ge \pi(B_*)+\frac{e^{-aw_k/2}}{2\sqrt k}
		\fstop
	\end{equation} Now,  for some $b> 0$ to be fixed later, define, for all $N\in \N$ large enough, 
	\begin{equation}
		A_*=A_*(N,a,b)\eqdef \set{\zeta \in \Omega_k: \sum_{x\in B_*}\zeta(x)\ge k\ttonde{\pi(B_*)+\tfrac{b}{\sqrt k}}}\fstop
	\end{equation}
	
	Let us continue from \eqref{eq:mixing-bin-lb1}. Since, under $\mu_{k,\eta}$, $k\in \N$ and $\eta\in \cP(\T^d_N)$, we have $Z^{k,\eta(B_*)}\eqdef\sum_{x\in B_*}\zeta(x)\sim {\rm Bin}(k,\eta(B_*))$, we obtain (recall $t=T(a)$)
	\begin{align}
		{\rm d}_k(t)&\ge \E\big[\mathsf P\big(Z^{k,\eta_t(0,B_*)}\ge k\ttonde{\pi(B_*)+\tfrac{b}{\sqrt k}}\big)\big] - \mathsf P\big(Z^{k,\pi(B_*)}\ge k\ttonde{\pi(B_*)+\tfrac{b}{\sqrt k}}\big)\\
		&\ge \E\big[\mathsf P\big(Z^{k,\eta_t(0,B_*)}\ge k\ttonde{\pi(B_*)+\tfrac{b}{\sqrt k}}\big)\big] - \frac{\pi(B_*)\tonde{1-\pi(B_*)}}{b^2}\\
		&\ge   \E\big[\mathsf P\big(Z^{k,\eta_t(0,B_*)}\ge k\ttonde{\pi(B_*)+\tfrac{b}{\sqrt k}}\big)\big] - \frac{1}{b^2}\comma
	\end{align}
	where for the second step we applied Cauchy-Schwarz inequality. Let us now provide a lower bound for the first term on the right-hand side above. For this purpose, for every $c>b>0$, 
	\begin{align}
		&\E\big[\mathsf P\big(Z^{k,\eta_t(0,B_*)}\ge k\pi(B_*)+b\sqrt k\big)\big]\\
		&\quad\ge \E\big[\mathsf P\big(Z^{k,\eta_t(0,B_*)}\ge k\pi(B_*)+b\sqrt k\big)\,\car_{\big\{\eta_t(0,B_*)\ge \pi(B_*)+\tfrac{c}{\sqrt k}\big\}}\big]\\
		&\quad\ge \mathsf P\big(Z^{k,\pi(B_*)+\frac{c}{\sqrt k}}\ge k\pi(B_*)+b\sqrt k\big)\,\P\big(\eta_t(0,B_*)\ge \pi(B_*)+\tfrac{c}{\sqrt k}\big)\\
		&\quad= \mathsf P\big(Z^{k,\pi(B_*)+\frac{c}{\sqrt k}}\ge k\pi(B_*)+c\sqrt k-(c-b)\sqrt k\big) \P\big(\eta_t(0,B_*)\ge \pi(B_*)+\tfrac{c}{\sqrt k}\big)\\
		&\quad\ge \frac{(c-b)^2}{(c-b)^2+
			1}\,	\P\big(\eta_t(0,B_*)\ge \pi(B_*)+\tfrac{c}{\sqrt k}\big)	\comma
	\end{align}
	where for the second inequality we used the stochastic domination $Z^{k,q}\gtrsim Z^{k,p}$ if $q>p$, while for the last one we used Cantelli inequality. Furthermore, provided that $e^{-aw_k/2}>4c$, we get, again by Cantelli inequality,
	\begin{align}
		\P\big(\eta_t(0,B_*)\ge \pi(B_*)+\tfrac{c}{\sqrt k}\big) &= \P\big(\eta_t(0,B_*)\ge \pi_t(0,B_*)-\big(\pi_t(0,B_*)-\pi(B_*)-\tfrac{c}{\sqrt k}\big)\big)\\
		&\ge \frac{1}{1+\sigma^2/r^2}\comma
	\end{align}
	with
	\begin{equation}
		\sigma\eqdef \E\big[\big(\eta_t(0,B_*)-\pi_t(0,B_*)\big)^2\big]^\frac12\comma\quad  r\eqdef 	\pi_t(0,B_*)-\pi(B_*)-\tfrac{c}{\sqrt k}\ge \frac{e^{-aw_k/2}/2-c}{\sqrt k}>0	\fstop
	\end{equation}
	Hence, we are done as soon as we can show that $\sigma^2/r^2\to 0$. Recall \eqref{eq:pit-pi-b} and $e^{-aw_k}>4c$; then, we have
	\begin{equation}
		r^2\ge \frac1k\tonde{\frac{e^{-aw_k/2}}{2}-c}^2\ge \frac{e^{-aw_k}}{16k} \fstop
	\end{equation}
	Moreover, recalling that $\|\frac{\eta_t(0,\emparg)}{\pi}-\frac{\pi_t(0,\emparg)}{\pi}\|_1= 2\sup_{B\subset \T^d_N}(\eta_t(0,B)-\pi_t(0,B))$, 	we get
	\begin{align}
		\sigma^2 \eqdef  \E\big[\big(\eta_t(0,B_*)-\pi_t(0,B_*)\big)^2\big]
		&\le \frac14\E\bigg[\bigg\|\frac{\eta_t(0,\emparg)}{\pi}-\frac{\pi_t(0,\emparg)}{\pi}\bigg\|_1^2\bigg]\\
		&\le \frac14 \E\bigg[\bigg\|\frac{\eta_t(0,\emparg)}{\pi}-\frac{\pi_t(0,\emparg)}{\pi}\bigg\|_2^2\bigg]\\
		&\le \frac{C}{N^2}\frac{e^{-aw_k}}{k}\exp\tonde{\frac{Bt_{\rm rel}}{2 N^{d+2}}\tonde{\log k +aw_k}}\comma
	\end{align}
	where for the second line we used Jensen inequality, while for the third one we used the second inequality in \eqref{eq:concentration} with $t=T(a)\ge N^2$. (Here, $B,C>0$ depend only on $d\ge 1$.)
	Hence, 
	\begin{equation}
		\frac{\sigma^2}{r^2}\le \frac{C}{N^2}\exp\tonde{\frac{Bt_{\rm rel}}{2N^{d+2}}\tonde{\log k+aw_k}}\le \frac{C}{N^2}\exp\tonde{\frac{Bt_{\rm rel}}{2N^{d+2}}\log k} \xrightarrow{N\to \infty} 0\comma
	\end{equation}
	provided that $\log k = o(N^d\log N)$.	Finally, setting $c>b>0$ large enough and $a<0$ small enough concludes the proof.
\end{proof}

\bibliographystyle{alpha}

\end{document}